\documentclass[11pt]{amsart}
\usepackage{amsxtra}
\usepackage{amssymb}
\input xy \xyoption{matrix} \xyoption{arrow}  \xyoption{curve}

\newcommand{\cleqn}{\setcounter{equation}{0}}
\newcommand{\clth}{\setcounter{theorem}{0}}
\newcommand {\sectionnew}[1]{\section{#1}\cleqn\clth}
\newcommand {\subsec} {\subsection}
\theoremstyle{plain}
\newtheorem{theorem}{Theorem}[section]
\newtheorem{lemma}[theorem]{Lemma}
\newtheorem{definition-theorem}[theorem]{Definition-Theorem}
\newtheorem{proposition}[theorem]{Proposition}
\newtheorem{corollary}[theorem]{Corollary}

\newtheorem*{propA*}{Proposition A}
\newtheorem*{thmB*}{Theorem B}
\newtheorem*{thmC*}{Theorem C}
\newtheorem*{thmD*}{Theorem D}
\newtheorem*{thmE*}{Theorem E}
\newtheorem*{thmF*}{Theorem F}
\newtheorem*{thmG*}{Theorem G}
\theoremstyle{definition}
\newtheorem{definition}[theorem]{Definition}

\newtheorem{example}[theorem]{Example}
\newtheorem{examples}[theorem]{Examples}

\newtheorem{remark}[theorem]{Remark}
\newtheorem*{remark*}{Remark}

\newtheorem{observation}[theorem]{Observation}

\newtheorem{problems}[theorem]{Problems}

\numberwithin{equation}{theorem}

\DeclareMathOperator{\End}{End} 

\newcommand\CC{\mathbb{C}}
\newcommand\QQ{\mathbb{Q}}
\newcommand\RR{\mathbb{R}}
\newcommand\ZZ{\mathbb{Z}}

\newcommand \Znn {\ZZ_{\ge 0}}

\newcommand \Zpos {\ZZ_{> 0}}
\DeclareMathOperator{\sr}{sr}

\newcommand\calC{\mathcal{C}}
\newcommand\calN{\mathcal{N}}
\newcommand\calT{\mathcal{T}}

\DeclareMathOperator{\FP}{FP}
\DeclareMathOperator{\FG}{FG}
\DeclareMathOperator{\Hom}{Hom}
\DeclareMathOperator{\Mod}{Mod}
\newcommand{\Modr}{{\Mod}{}\text{-}R}
\newcommand{\rMod}{R\text{-}{\Mod}}
\DeclareMathOperator{\card}{card}

\begin{document}

\title{Levels of cancellation for monoids and modules}

\author[Ara]{Pere Ara}
\address{Departament de Matem\`atiques, Universitat Aut\`onoma de Barcelona, 08193 Bella\-terra, Barcelona, Spain, and Barcelona Graduate School of Mathematics (BGSMath), Campus de Bellaterra, Edifici C, 08193 Bellaterra, Barcelona, Spain}
\email{para@mat.uab.cat}

\author[Goodearl]{\ Ken  Goodearl}
\address{Department of Mathematics, University of California, Santa Barbara, CA 93106, USA}
\email{goodearl@math.ucsb.edu}

\author[Nielsen]{\ Pace P.\ Nielsen} \
\address{Department of Mathematics, Brigham Young University, Provo, UT 84602, USA}
\email{pace@math.byu.edu}

\author[O'Meara]{\ Kevin C. O'Meara}
\address{University of Canterbury, Christchurch, New Zealand}
\email{staf198@uclive.ac.nz}

\author[Pardo]{\ Enrique Pardo}
\address{Departamento de Matem\'aticas, Facultad de Ciencias, Universidad de C\'adiz, Campus de Puerto Real, 11510 Puerto Real (C\'adiz), Spain}
\email{enrique.pardo@uca.es}

\author[Perera]{\ Francesc Perera}
\address{Departament de Matem\`atiques, Universitat Aut\`onoma de Barcelona, Bellaterra, Spain, and Centre de Recerca Matem\`atica, Bellaterra, Spain}
\email{perera@mat.uab.cat}

\begin{abstract} 
Levels of cancellativity in commutative monoids $M$, determined by stable rank values in $\Zpos \cup \{\infty\}$ for elements of $M$, are investigated.  The behavior of the stable ranks of multiples $ka$, for $k \in \Zpos$ and $a \in M$, is determined.  In the case of a refinement monoid $M$, the possible stable rank values in archimedean components of $M$ are pinned down.  Finally, stable rank in monoids built from isomorphism or other equivalence classes of modules over a ring is discussed.
\end{abstract}



\maketitle

\sectionnew{Introduction}  \label{intro}

We study the cancellative behavior of elements in commutative monoids, as determined by their \emph{stable ranks}, which are values in $\Zpos \cup \{\infty\}$ modelled on cancellation conditions for direct sums of modules in algebraic K-theory tied to the concept of stable ranks of (endomorphism) rings.  Cancellation in monoids is thus stratified in different levels: the higher the stable rank of an element, the more restrictive the level in which it cancels.  The essence of the condition is that if an element $a$ in a commutative monoid $M$ has finite stable rank, say stable rank $n$, then $a+x=a+y$ implies $x=y$ for any elements $x,y \in M$ such that $na$ is a summand of $x$ (Lemma \ref{W1.2}).

K-theoretic cancellation for modules was established by Evans \cite{Ev} (for stable rank $1$) and Warfield \cite{War} (for general stable rank).  Namely: If $A$ is a module over a ring $R$ and the endomorphism ring $\End_R(A)$ has finite K-theoretic stable rank (recalled in Definition \ref{sr.ring}), then $A$ cancels from certain direct sums $A \oplus B \cong A \oplus C$.  In case $\End_R(A)$ has stable rank $1$, then $A \oplus B \cong A \oplus C$ implies $B \cong C$ for arbitrary $R$-modules $B$ and $C$ \cite[Theorem 2]{Ev}, while if $\End_R(A)$ has stable rank $n \ge 2$, then $A \oplus B \cong A \oplus C$ implies $B \cong C$ for $B$ and $C$ such that $A^n$ is isomorphic to a direct summand of $B$ \cite[Theorems 1.6, 1.2]{War}.  Warfield proved that the condition $\sr(\End_R(A)) = n$ is equivalent to a certain ``$n$-substitution property" \cite[Theorem 1.6]{War}, which has the following consequence: If $A \oplus B \cong A \oplus C$ and $B \cong A^{n-1} \oplus B'$, then there is a module $E$ such that $A^n \cong A \oplus E$ and $B' \oplus E \cong C$ \cite[Theorem 1.3]{War}.

Suppose $M$ is a monoid whose elements are (labels for) the isomorphism classes $[X]$ in some class $\calC$ of $R$-modules closed under pairwise direct sums and whose operation is given by direct sums: $[X]+[Y] = [X \oplus Y]$.  Assuming $\calC$ contains $A$, $B$, $B'$, $C$ and all direct summands of $A^n$, the above consequence of the $n$-substitution property for $\End_R(A)$ reads
$$
n[A] + [B'] = [A] + [C] \ \implies \ \exists\ [E] \ \text{such that} \ n[A] = [A] + [E] \ \text{and} \ [E] + [B'] = [C]
$$
in $M$.
This condition provides the definition of stable rank for an element $a$ in an arbitrary commutative monoid $M$ \cite[p.122]{AGOP}:
\begin{itemize}
\item $\sr_M(a)$ is the least positive integer $n$ such that
$$
na+x = a+y \ \implies \ \exists\ e \in M \ \text{such that} \ na = a+e \ \text{and} \ e+x = y
$$
for any $x,y \in M$ (if such $n$ exist), or $\infty$ (otherwise).
\end{itemize}
In a monoid of isomorphism classes of the type mentioned, $\sr_M([A]) \le \sr(\End_R(A))$ by Warfield's theorem.  In case $R$ is an exchange ring and $M$ is the monoid $V(R)$ of isomorphism classes of finitely generated projective $R$-modules, there is an equality of stable ranks: $\sr_M([A]) = \sr(\End_R(A))$ for all $A \in M$ \cite[Theorem 3.2]{AGOP}.

Stable ranks of elements of refinement monoids have been investigated in \cite{AGOP} and subsequent works such as \cite{Ara.stability.survey, AGalmost, AGtamewild, MDS, Par}, for application to von Neumann regular and/or exchange rings.  The present work is an investigation of stable ranks in general commutative monoids, with further development for refinement monoids.
\medskip

Throughout, all monoids will be commutative, written additively.

\subsec{Contents}
Section \ref{sr.cond} contains definitions, examples, and basic properties of stable ranks of monoid elements.  Stable rank values in quotients and o-ideals are studied in Section \ref{sr.quo}.  Section \ref{sr.mult} contains our main results on stable ranks of multiples:

\begin{propA*}
{\rm(Lemma \ref{sr.decrease}, Proposition \ref{sr=n->na.hermite})}
Let $M$ be a monoid and $a \in M$.

{\rm(1)} The stable ranks of the multiples $ka$ decrease as $k$ increases: $\sr_M(ka) \ge \sr_M(la)$ for all positive integers $k \le l$.

{\rm(2)} If $\sr_M(a)$ is finite, the stable ranks of the $ka$ are eventually $\le2$, namely for $k \ge \sr_M(a)-1$.
\end{propA*}

\begin{thmB*}
{\rm(Theorems \ref{W1.12}, \ref{sr.ka.refine})}
Let $M$ be a monoid and $a \in M$ with $\sr_M(a) = n < \infty$.  Then
$$
1 + \left\lfloor \frac{n-1}{l} \right\rfloor \le \sr_M(la) \le 1 + \left\lceil \frac{n-1}{l} \right\rceil
$$
for all positive integers $l$.  Further, $\sr_M(la) = 1 + \left\lceil \frac{n-1}{l} \right\rceil$ in case $M$ is a refinement monoid.
\end{thmB*}

Section \ref{sr.value.arch} concerns stable rank values in archimedean components and in separative monoids.

\begin{thmC*}
{\rm(Theorem \ref{srC})}
Let $C$ be an archimedean component in a monoid $M$.  

{\rm(1)} The set $\sr_M(C) := \{ \sr_M(c) \mid c \in C \}$ equals $\Zpos$ or $\ZZ_{\ge2}$ or $\{\infty\}$ or a finite subset of $\Zpos$.

{\rm(2)} If $M$ is conical, then $\sr_M(C)$ equals $\{1\}$ or $\ZZ_{\ge2}$ or $\{\infty\}$ or a finite subset of $\ZZ_{\ge2}$.
\end{thmC*}

\begin{thmD*}
{\rm(Corollary \ref{trichot})}
If $M$ is a separative monoid, then the stable rank of any element of $M$ is $1$, $2$, or $\infty$.
\end{thmD*}

Stable rank values in conical refinement monoids are investigated in Sections \ref{sr.refmon} and \ref{sr.Zge2}.

\begin{thmE*}
{\rm(Theorems \ref{simple.crm.bdd.sr}, \ref{trichot.simple.crm}, \ref{simple.crm.sr[2,infty)})}
Let $M$ be a simple conical refinement monoid.

{\rm(1)} If the stable ranks of the elements of $M$ have a finite upper bound, then $M$ is cancellative, and all its elements have stable rank $1$.

{\rm(2)} The set $\sr_M(M\setminus \{0\})$ equals $\{1\}$ or $\{\infty\}$ or $\ZZ_{\ge2}$.  All three possibilities occur.
\end{thmE*}

In Sections \ref{V(R)s} and \ref{V(C)s}, we discuss monoids built from isomorphism or other equivalence classes of certain types of modules, stable rank values in these monoids, and relations with the K-theoretic stable ranks of the endomorphism rings of the modules concerned.

\subsec{Terminology and notation}
We repeat our general hypothesis, that all monoids in this paper are commutative and additive.

The \emph{units} of a monoid $M$ are the elements that are invertible (with respect to addition).  Denote the set of units of $M$ by $U(M)$; it is an abelian group.
The monoid $M$ is \emph{conical} (or \emph{reduced}) if $x+y=0$ implies $x=y=0$ for any $x,y \in M$, that is, if $U(M) = \{0\}$.  

The \emph{algebraic ordering} on $M$ is the reflexive, transitive relation $\le$ defined by 
$$
x \le y \ \iff \ \exists\ z \in M \ \text{such that} \ x+z = y.
$$
An element $u \in M$ is an \emph{order-unit} if
$$
\forall\ x \in M, \ \ \exists\ n \in \Zpos\, \ \text{such that}\, \ x \le nu.
$$
An \emph{o-ideal} of $M$ is a submonoid $I$ such that $x \le y \in I$ implies $x \in I$ for any $x,y \in M$.  We write $\langle x \rangle$ for the o-ideal generated by an element $x \in M$, that is,
$$
\langle x \rangle := \{ y \in M \mid y \le nx\, \ \text{for some}\, \ n \in \Zpos \rangle.
$$
We say that $M$ is \emph{simple} if it has precisely two o-ideals, that is, $M$ is not a group and its only o-ideals are $U(M)$ and $M$.  In the conical case, these conditions amount to requiring that $M$ is nonzero and all its nonzero elements are order-units.

Elements $x,y \in M$ are \emph{asymptotic}, written $x \asymp y$, if $\langle x\rangle = \langle y\rangle$, that is, if there exist $m,n \in \Zpos$ such that $x \le my$ and $y \le nx$.  The relation $\asymp$ is an equivalence relation on $M$, the equivalence classes of which are called \emph{archimedean components}.  We write $M(x)$ for the archimedean component containing an element $x$ of $M$.

Given an o-ideal $I$ of $M$, there is a congruence $\equiv_I$ on $M$ given by the rule
$$
x \equiv_I y \ \iff \ \exists\ a,b \in I\, \ \text{such that}\, \ x+a = y+b.
$$
The monoid $M/{\equiv_I}$ is called the \emph{quotient of $M$ modulo $I$} and is denoted $M/I$.  We write $[x]_I$, or just $[x]$ if $I$ is understood, for the $\equiv_I$-equivalence class of $x$ in $M/I$.

The (\emph{Riesz}) \emph{refinement property} for $M$ is the condition
\begin{multline*}
\forall\ x_1,x_2,y_1,y_2 \in M, \quad x_1+x_2 = y_1+y_2 \ \implies \ \exists\ z_{ij} \in M\, \ \text{such that}  \\
 x_i = z_{i1}+z_{i2}\, \ \text{for}\, \ i=1,2\, \ \text{and}\, \ y_j = z_{1j}+z_{2j}\, \ \text{for}\, \ j= 1,2.
 \end{multline*}
When this holds, $M$ \emph{has refinement}, or is a \emph{refinement monoid}.  The refinement property implies analogous refinements for all equations $\sum_{i=1}^m x_i = \sum_{j=1}^n y_j$ in $M$.

The monoid $M$ is \emph{cancellative} if $x+z=y+z$ implies $x=y$ for any $x,y,z \in M$, and it is \emph{separative} (or \emph{has separative cancellation}) if $2x=x+y=2y$ implies $x=y$ for any $x,y \in M$.  We say that $M$ is \emph{strongly separative} if $2x=x+y$ implies $x=y$ for any $x,y \in M$.

Several conditions equivalent to separativity were given in \cite[Lemma 2.1]{AGOP}, which we state here for reference.

\begin{lemma}  \label{agop2.1}
For a monoid $M$, the following conditions are equivalent:

{\rm(a)} $M$ is separative.

{\rm(b)} For any $x,y \in M$, if $2x=2y$ and $3x=3y$, then $x=y$.

{\rm(c)} For any $x,y \in M$ and $n \in \Zpos$, if $nx=ny$ and $(n+1)x=(n+1)y$, then $x=y$.

{\rm(d)} For any $x,y,z \in M$, if $x+z = y+z$ and $z \in \langle x \rangle \cap \langle y \rangle$, then $x=y$.

In case $M$ has refinement, separativity is also equivalent to the following:

{\rm(e)} For any $x,y,z \in M$, if $x+2z = y+2z$, then $x+z = y+z$.
\end{lemma}

There are similar equivalent conditions for strong separativity:

\begin{lemma}  \label{agop.p126}
For a monoid $M$, the following conditions are equivalent:

{\rm(a)} $M$ is strongly separative.

{\rm(b)} For any $x,y \in M$ and $n \in \Zpos$, if  $(n+1)x= nx+y$, then $x=y$.

{\rm(c)} For any $x,y,z \in M$, if $x+z = y+z$ and $z \in \langle x \rangle$, then $x=y$.

{\rm(d)} For any $x,y,z \in M$, if $x+2z = y+z$, then $x+z=y$.
\end{lemma}

\begin{proof}
The equivalence of (a), (c), and (d) was noted in \cite[p.126]{AGOP}, the straightforward proofs being monoid forms of arguments for direct sums of modules given/noted in \cite[Proposition 4.2]{AOT} and \cite[Lemma 5.1]{AGOP}.

(c)$\Longrightarrow$(b): Given $(n+1)x= nx+y$, we have $x+nx = y+nx$.  Since $nx \in \langle x \rangle$, condition (c) implies $x=y$.

(b)$\Longrightarrow$(c): Suppose $x+z = y+z$ and $z \in \langle x \rangle$.  Then $z+z' = nx$ for some $z' \in M$ and $n \in \Zpos$.  Add $z'$ to both sides of $x+z = y+z$ to get $(n+1)x = nx+y$.  Condition (b) then implies $x=y$.
\end{proof}

\sectionnew{Stable rank conditions}  \label{sr.cond}

Fix a commutative monoid $M$ throughout this section.  We recall the definition of stable rank for elements of $M$ from \cite[p.122]{AGOP}, give examples exhibiting all possible values, and develop some basic properties of stable rank.

\begin{definition}  \label{def.sr}
Let $a \in M$ and $n \in \Zpos$.  We say that $a$ satisfies the \emph{$n$-stable rank condition} (in $M$) if whenever $na+x= a+y$ for some $x,y \in M$, there exists $e\in M$ such that $na=a+e$ and $e+x= y$. Observe that the $n$-stable rank condition implies the $m$-stable rank condition for all $m > n$.

The \emph{stable rank of $a$} (in $M$), denoted $\sr_M(a)$ or just $\sr(a)$ if $M$ is understood, is the least positive integer $n$ such that $a$ satisfies the $n$-stable rank condition (provided such an $n$ exists) or $\infty$ (otherwise).  

The value $\sr_M(a)$ may vary if $M$ is changed.  For instance, if $a$ cancels from sums within the submonoid $A := \{ na \mid n \in \Znn \}$, then $\sr_A(a) = 1$, regardless of the value of $\sr_M(a)$.
\end{definition}

An initial cache of examples shows that all allowable values of stable rank occur:

\begin{examples}  \label{exist.sra=?}
(1) First, take $M := \Znn \sqcup \{\infty\}$.  In this monoid, $\sr(a) = 1$ for all $a \in \Znn$, since such elements $a$ cancel from sums in $M$.  On the other hand, $\sr(\infty) = \infty$, since for any $n \in \Zpos$ we have $n\cdot\infty + \infty = \infty+0$, but there is no $e\in M$ satisfying $e+\infty = 0$.

(2) Now let $n$ be an integer $\ge2$, and let $M$ be the commutative monoid presented by two generators, $a$ and $b$, and two relations, $na = a+b$ and $2(n-1)a = 2b$.  It follows from the relations that every element of $M$ equals either $b$ or a nonnegative multiple of $a$.  Moreover, the elements $0,a,b,2a,3a,\dots$ in $M$ are all distinct.  $\bigl($There is a monoid homomorphism $M \rightarrow \Znn$ sending $a \mapsto 1$ and $b \mapsto n-1$, from which we see that $0,a,2a,3a,\dots$ are distinct and $b \ne 0$.  Also, there is a $3$-element monoid $M' := \{0,x,\infty\}$ such that $2x = \infty$, and there is a monoid homomorphism $M \rightarrow M'$ sending $a \mapsto \infty$ and $b \mapsto x$.  From this we see that $b \ne ma$ for all $m \in \Znn$.$\bigr)$ Note for later use that it follows that $M$ is conical.

Observe that $(n-1)a + a = a + b$, but there is no $e \in M$ with $(n-1)a = a+e$ and $e+a = b$.  Consequently, $\sr(a) \nleq n-1$.

Suppose $na+x = a+y$ for some $x,y \in M$.  If $x = 0$, we have $na = a+y$ and $y+x = y$.  If $x \ne 0$, either $x = b$ or $x = ma$ for some $m \in \Zpos$.  In both cases, we check that $y = (n-1)a + x$.  Since also $na = a + (n-1)a$, we conclude that $\sr(a) \le n$.  Therefore $\sr(a) = n$.

For later reference, we record that $\sr(b) = 2$.  On one hand, $b+b = b+(n-1)a$, but there is no $e \in M$ with $b = b+e$ and $e+b = (n-1)a$.  Thus, $\sr(b) >1$.  On the other hand, if $2b+x = b+y$ for some $x,y \in M$, we find that the pair of equations $2b = b+e$ and $e+x = y$ can be solved with $e := b$ (if $y=b$) or with $e := (n-1)a$ (otherwise).

(3) A different example with elements of stable rank $2$ will also be useful for later reference.  This time, let $M$ be the commutative monoid presented by two generators, $a$ and $b$, and one relation, $a+b=a$.  Clearly, every element of $M$ is either a positive multiple of $a$ or a nonnegative multiple of $b$.  These elements, namely $0,a,b,2a,2b, \dots$, are all distinct, as follows.  On one hand, there is a homomorphism $M \rightarrow \Znn$ sending $a \mapsto 1$ and $b \mapsto 0$.  Consequently, $0,a,2a,\dots$ are all distinct, and $ma \ne nb$ for all $m,n \in \Zpos$.  On the other hand, there is a homomorphism $M \rightarrow \Znn \sqcup \{\infty\}$ sending $a \mapsto \infty$ and $b \mapsto 1$, whence $0,b,2b,\dots$ are all distinct.

We claim that $\sr(ma) = 2$ for any $m \in \Zpos$.

Note that $ma+b = ma+0$ but there is no $e \in M$ satisfying $e+b = 0$.  Thus, $\sr(ma) \ne 1$.

Now consider $x,y \in M$ such that $2ma+x = ma+y$.  If $x$ is a multiple of $b$, this equation reduces to $2ma = ma+y$.  Since $2ma \ne ma$, we find that $y = ma = ma+x$.  If $x = na$ for some $n \in \Zpos$, then $(2m+n)a = ma+y$.  In this case, we find that $y = (m+n)a$, and again $y = ma+x$.  In both cases, $e := ma$ is a solution for $2ma = ma+e$ and $e+x = y$.  Therefore $\sr(ma) = 2$.
\end{examples}

The easy argument that the $n$-stable rank condition implies the $(n+1)$-stable rank condition also yields the following cancellation result, analogous to \cite[Theorem 1.2]{War}.

\begin{lemma}  \label{W1.2}
Let $a \in M$ with $\sr(a) \le n < \infty$.  If $(n+1)a + b = a + c$ for some $b,c \in M$, then $na + b = c$. Equivalently, if $a + d = a + c$ for some $d,c \in M$ with $na \le d$, then $d = c$.
\end{lemma} 

\begin{proof}
Given $(n+1)a + b = a + c$, we have $na + (a + b) = a + c$, so $\sr(a) \le n$ implies that there is some $e \in M$ with $na = a + e$ and $e + (a + b) = c$.  Combining these two equations yields $na + b = c$.  The equivalence with the second condition is clear.
\end{proof}

\begin{theorem}  \label{W1.9}
If $a = a_1 + \cdots + a_t$ for some $a_i \in M$, then
$$
\sr(a) \le \max(\sr(a_1),\dots,\sr(a_t)).
$$
\end{theorem}

\begin{proof}
It suffices to deal with the case when $t = 2$ and $n := \max(\sr(a_1), \sr(a_2)) < \infty$.  

Suppose $n a + x = a + y$ for some $x,y \in M$.  Since $n a_1 + (n a_2 + x) = a_1 + (a_2 + y)$ and $\sr(a_1) \le n$, there is some $e_1 \in M$ such that $n a_1 = a_1 + e_1$ and $e_1 + (n a_2 + x) = a_2 + y$.  Then $\sr(a_2) \le n$ implies that there is some $e_2 \in M$ such that $n a_2 = a_2 + e_2$ and $e_2 + (e_1 + x) = y$.  Setting $e := e_1 + e_2$, we conclude that $n a = a + e$ and $e + x = y$.  Therefore $\sr(a) \le n$.
\end{proof}

\begin{corollary}  \label{sr+unit}
If $a,b \in M$ and $b$ is a unit, then $\sr(a+b) = \sr(a)$.
\end{corollary}

\begin{proof}
Since $b$ is cancellative, $\sr(b) = 1\le \sr(a)$, and so $\sr(a+b) \le \sr(a)$.  The reverse inequality follows in the same way, because $a = (a+b)+b'$ for some unit $b'$.
\end{proof}

Theorem \ref{W1.9} of course implies that if $a \in M$ with $\sr(a) \le n$, then $\sr(ka) \le n$ for all $k \in \Zpos$.  In the case $n=1$, the conclusion $\sr(ka) = 1$ can be slightly improved:

\begin{lemma}  \label{sr(ka)=1}
Let $a \in M$ with $\sr(a) = 1$ and $k \in \Zpos$.  If $x,y \in M$ with $ka + x = ka + y$, there exists $e \in M$ with $a = a + e$ and $e + x = y$.
\end{lemma}

\begin{proof}
By Lemma \ref{W1.2}, $ka + x = ka + y$ implies $a+x = a+y$.  The result follows.
\end{proof}

Low stable rank is closely connected to the following conditions.

\begin{definition}  \label{def.cancel.etc}
An element $a \in M$ is
\begin{align*}
&\text{\emph{cancellative}} &&\text{in case} &\quad a+x &= a+y \ \implies \ x = y, &&\forall\; x,y \in M;  \\
&\text{\emph{Hermite}} &&\text{in case} &\quad 2a+x &= a+y \ \implies \ a+x = y, &&\forall\; x,y \in M;  \\
&\text{\emph{self-cancellative}} &&\text{in case} &\quad 2a &= a+y \ \implies \ a = y, &&\forall\; x,y \in M.
\end{align*}
Note that $M$ is strongly separative if and only if all its elements are self-cancellative.
\end{definition}

\begin{proposition}  \label{hermite.etc}
Let $a \in M$.

{\rm(a)} If $a$ is cancellative, then $\sr(a) = 1$.

{\rm(b)} If $\sr(a) = 1$, then $a$ is Hermite.

{\rm(c)} If $a$ is Hermite, then it is self-cancellative and $\sr(a) \le 2$.

{\rm(d)} If all elements of the set $a+M$ are self-cancellative in $M$, then $a$ is Hermite.  Consequently, if $M$ is strongly separative, then all its elements are Hermite.

{\rm(e)} Assume that $\sr(a) = 1$. If $x,y \in M$ with $a+x = a+y$, there is a unit $e \in M$ such that $a = a+e$ and $e+x = y$.

{\rm(f)} Assuming $M$ is conical, then $\sr(a) = 1$ if and only if $a$ is cancellative.
\end{proposition}

\begin{proof}
(a), (b), (c) are clear from the definitions and Lemma \ref{W1.2}, and (f) follows immediately from (a), (e).

(d) Suppose that $2a+x = a+y$ for some $x,y \in M$.  Then $2(a+x) = (a+x)+y$, and self-cancellativity of $a+x$ implies that $a+x = y$, proving that $a$ is Hermite.  The stated consequence is immediate.

(e) First, $\sr(a) = 1$ implies that there exists $e \in M$ such that $a = a+e$ and $e+x = y$.  Writing $a+e = a+0$ and applying $\sr(a) = 1$ a second time, there exists $e' \in M$ with $a=a+e'$ and $e' +e=0$. Therefore $e$ is a unit in $M$.
\end{proof}

\begin{examples}  \label{sr2.not.hermite}
(1) The conical hypothesis of Proposition \ref{hermite.etc}(f) cannot be dropped. For example, take $M = A \sqcup B$ where $A$ is a nonzero abelian group, $B = (\Zpos,+)$, and $a + b = b$ for all $a \in A$ and $b \in B$.  Any $b \in B$ fails to be cancellative, since $A$ is nonzero, but $sr(b) = 1$ holds, as follows.  Suppose $b + x = b + y$ for some $x,y \in M$.  If one of $x$ or $y$ is in $A$, so is the other (since $b+b' \ne b$ for any $b' \in B$), and hence $b = b + e$ and $e + x = y$ by taking $e := y - x \in A$.  Otherwise, $x,y \in B$, whence $x = y$ and so taking $e := 0$ yields $b = b + e$ and $e + x = y$.

(2) We observe further that self-cancellativity of an element $a \in M$ does not imply that either $a$ is Hermite or $\sr(a) \le 2$.  Indeed, take $M$ to be presented by a single element $a$ subject to the relation $3a = a$; so $M = \{0, a, 2a\}$.  Here $a$ is self-cancellative since the only solution to $2a = a+y$ is $y = a$. However, $a$ is not Hermite since $2a+a = a+0$ but $a + a \ne 0$, and $\sr(a) \nleq 2$ since $2a+a = a+0$ but there is no $e \in M$ with $e+a = 0$.  In fact, $\sr(a) = \infty$, as the following lemma shows.
\end{examples}

\begin{lemma}  \label{purelyinf}
If $a \in M$ is a non-unit and $(k+1)a \le ka$ for some positive integer $k$, then $\sr(a) = \infty$.
\end{lemma}

\begin{proof} 
It follows from $(k+1)a \le ka$ that $ma \le ka$ for all $m>k$.  Given a positive integer $n$, we have $(k+n)a + x = ka$ for some $x \in M$ (depending on $n$).  If $\sr(a) \le n$, then since $ka + (na+x) = ka+0$, Lemma \ref{W1.2} implies $na+x = 0$.  However, this is impossible because $a$ is not a unit.  Therefore $\sr(a) > n$ for all $n$.
\end{proof}

When $M$ has refinement, the requirements for an element of $M$ to have stable rank at most $n$ can be reduced, as follows.  The argument was given by Ara for monoids of projective modules \cite[Theorem 2.2]{Ara.stability.survey} and noted in general (unpublished correspondence).

\begin{proposition}  \label{srequiv.refmon}
Assume $M$ has refinement.  Let $a \in M$ and $n \in \Zpos$.  Then the following conditions are equivalent:

{\rm(a)} $\sr(a) \le n$.  

{\rm(b)}  Whenever $na+x = a+y$ for some $x,y \in M$, then $x \le y$.

{\rm(c)} Whenever $na+x = a+y$ for some $x,y \in M$ with $x \le a$ and $y \le na$, then $x \le y$.

{\rm(d)} Whenever $na = u+v$ and $a = u+w$ for some $u,v,w \in M$, then $w \le v$.
\end{proposition}

\begin{proof}
(a)$\Longrightarrow$(b) and (b)$\Longrightarrow$(c): Obvious.

(c)$\Longrightarrow$(d): From $na = u+v$ and $a = u+w$, we immediately get $a+v = na+w$.  Applying (c) with $x :=w \le a$ and $y :=v \le na$, we obtain $w \le v$.

(d)$\Longrightarrow$(a): Suppose that $na+x = a+y$ for some $x,y \in M$.  By refinement, there are decompositions $na = a_1+a_2$ and $x = x_1+x_2$ for some $a_i,x_j \in M$ with $a = a_1+x_1$ and $y = a_2+x_2$.  Then by (d), $x_1 \le a_2$, so $a_2 = x_1+e$ for some $e \in M$.  Now we have
\begin{align*}
na &= a_1+a_2 = a_1 +x_1+e = a+e  \\
y &= a_2+x_2 = x_1+e+x_2 = e+x,
\end{align*}
verifying that $\sr(a) \le n$.
\end{proof}

\sectionnew{Quotients}  \label{sr.quo}

Stable rank typically behaves poorly in the passage from a monoid to a quotient.  For example, any commutative monoid $M$ is a quotient of some direct sum $\Znn^{(I)}$, and the elements of $\Znn^{(I)}$ have stable rank $1$ while the stable ranks of elements of $M$ can be arbitrary.  However, there are certain quotients for which stable rank is reasonably well behaved, as we show in this section.  We continue to fix a commutative monoid $M$.

\begin{lemma}  \label{srM/I}
If $I$ is an o-ideal of $M$ and $a \in M$, then $\sr_{M/I}([a]) \le \sr_M(a)$.
\end{lemma}

\begin{proof}
Suppose $\sr_M(a) = n < \infty$, and consider $x,y \in M$ such that $n[a]+[x] = [a]+[y]$ in $M/I$.  Then $na+x+u = a+y+v$ for some $u,v \in I$.  Hence, there is some $e \in M$ with $na = a+e$ and $e+x+u = y+v$, so $n[a] = [a]+[e]$ and $[e]+[x] = [y]$.  Thus $\sr_{M/I}([a]) \le n$.
\end{proof}

\begin{lemma}  \label{srI.srM}
Let $I$ be an o-ideal of $M$ and $a \in I$.

{\rm(a)}  $\sr_I(a) \le \sr_M(a)$.

{\rm(b)}  If $M$ has refinement, then $\sr_I(a) = \sr_M(a)$.

{\rm(c)} If $M$ has refinement and $a$ is Hermite within $I$, then $a$ is Hermite in $M$.
\end{lemma}

\begin{proof}
(a) Suppose $\sr_M(a) = n < \infty$, and consider $x,y \in I$ such that $na+x = a+y$.  Then there exists $e \in M$ such that $na = a+e$ and $e+x = y$.  Since $e \le na$ (or since $e \le y$), we have $e \in I$.  Thus, $\sr_I(a) \le n$.

(b) Suppose $\sr_I(a) = m < \infty$, and consider $x,y \in M$ such that $ma+x = a+y$.  Then $ma = a_1+a_2$ and $x = x_1+x_2$ for some $a_i,x_j \in M$ such that $a_1+x_1 = a$ and $a_2+x_2 = y$.  Since $x_1 \le a$ and $a_2 \le ma$, we have $x_1,a_2 \in I$.  Moreover,
$$
ma + x_1 = a_1 + a_2 + x_1 = a + a_2 \,,
$$
and so there exists $e \in I$ with $ma = a+e$ and $e+x_1 = a_2$.  Since 
$$
e+x = e+x_1+x_2 = a_2+x_2 = y,
$$
we conclude that $\sr_M(a) \le m$.

(c) This is proved in the same manner as (b).
\end{proof}

In certain quotients by congruences, stable ranks can be controlled.  A rather trivial example is the \emph{stable equality} congruence, given by
$$
u \equiv v \ \iff \ u+w = v+w\ \; \text{for some}\ \; w \in M.
$$
For this congruence, $M/{\equiv}$ is cancellative, whence $\sr([a]_{\equiv}) = 1$ for all $a \in M$.

We present three other instances.  In the first, the quotient $M/{\equiv}$ is known as the \emph{maximal antisymmetric quotient of $M$}, antisymmetry being taken with respect to the algebraic order.  The other two examples concern congruences modelled on near-isomorphism and multi-isomorphism of abelian groups (see Examples \ref{stC/sim.expls}).

The monoid $M$ is said to be \emph{stably finite} provided $a+x=a \implies x=0$, for any $a,x \in M$.

\begin{proposition}  \label{max.antisymm.quo}
Let $\equiv$ be the congruence on $M$ defined by
$$
x \equiv y \ \iff \ x \le y \le x.
$$

{\rm(a)}  If $M$ is a stably finite refinement monoid, then so is $M/{\equiv}$.

{\rm(b)}  $\sr_{M/{\equiv}}([a]_\equiv) \le \sr_M(a)+1$ for all $a \in M$.

{\rm(c)}  If $M$ is stably finite or $M/{\equiv}$ has refinement, then $\sr_{M/{\equiv}}([a]_\equiv) \le \sr_M(a)$ for all $a \in M$.

{\rm(d)} If $M$ has refinement, then $\sr_{M/{\equiv}}([a]_\equiv) \ge \sr_M(a)$ for all $a \in M$.

{\rm(e)}  If $M$ is a stably finite refinement monoid, then $\sr_{M/{\equiv}}([a]_\equiv) = \sr_M(a)$ for all $a \in M$.
\end{proposition}

\begin{proof}
We abbreviate $[-]_{\equiv}$ to $[-]$ throughout the proof.

(a) Suppose $[a]+[b] = [a]$ for some $a,b \in M$.  Then $a+b+c = a$ for some $c \in M$, and stable finiteness implies $b+c=0$.  In particular, $b \le 0 \le b$, and thus $[b] = [0]$.  This shows that $M/{\equiv}$ is stably finite.

Note that when $x,y \in M$ with $x \equiv y$, we have $x+a=y$ and $y+b=x$ for some $a,b \in M$, whence $x+a+b = x$, so stable finiteness implies $a+b=0$.  Thus, we can say that $x \equiv y$ if and only if $x = y+b$ for some unit $b \in M$.

Refinement was proved in \cite[Proposition 2.4]{MDS} under the assumption that $M$ is cancellative.  We utilise the same argument.

Suppose that $[a_0]+[a_1] = [b_0]+[b_1]$ in $M/{\equiv}$ for some $a_i,b_j \in M$.  There is a unit $b_2 \in M$ such that 
$$
a_0+a_1 = b_0+b_1+b_2 \,.
$$
By refinement, there are elements $c_{ij} \in M$ such that
$$
a_i = c_{i0}+c_{i1}+c_{i2} \quad \forall\; i=0,1 \qquad \text{and} \qquad b_j = c_{0j}+c_{1j} \quad \forall\; j=0,1,2.
$$
Since $b_2$ is a unit, so are $c_{02}$ and $c_{12}$.  Consequently,
$[a_i] = [c_{i0}]+[c_{i1}]$ for $i=0,1$.
Since also $[b_j] = [c_{0j}]+[c_{1j}]$ for $j=0,1$, refinement in $M/{\equiv}$ is established.

(b) Let $a \in M$ with $\sr_M(a) = n < \infty$, and consider $x,y \in M$ such that $(n+1)[a]+[x] = [a]+[y]$.  On one hand, $(n+1)a+x+b = a+y$ for some $b \in M$, whence Lemma 2.3 implies $na+x+b = y$, and so $na+x \le y$.  On the other hand, $(n+1)a+x = a+y+c$ for some $c \in M$, whence Lemma 2.3 implies $na+x = y+c$, yielding $y \le na+x$.  Consequently, $n[a]+[x] = [y]$, which means we can solve $(n+1)[a] = [a]+e$ and $e+[x] = [y]$ with $e := n[a]$.  Thus, $\sr_{M/{\equiv}}([a]) \le n+1$.

(c)  Suppose that $\sr_M(a) = n < \infty$; we need to show that $\sr_{M/{\equiv}}([a]) \le n$.

If $M$ is stably finite and $n[a]+[x] = [a]+[y]$ for some $x,y \in M$, then, as noted in (a), $na+x+b = a+y$ for some unit $b \in M$.  Now $\sr_M(a) \le n$ implies the existence of some $e \in M$ such that $na = a+e$ and $e+x+b = y$.  Then $n[a] = [a]+[e]$ and, since $b$ is a unit, $[e]+[x] = [y]$.  This establishes $\sr_{M/{\equiv}}([a]) \le n$ in the stably finite case.

Assume now that $M/{\equiv}$ has refinement.  Proposition \ref{srequiv.refmon} says that in order to prove $\sr_{M/{\equiv}}([a]) \le n$, it suffices to show that whenever $x,y \in M$ with $n[a]+[x] = [a]+[y]$, we have $[x] \le [y]$.  Now $na+x+b = a+y$ for some $b \in M$, and $\sr_M(a) \le n$ implies there is some $e \in M$ for which $e+x+b = y$, which yields $[x]+[e+b] = [y]$ and $[x] \le [y]$ as required.

(d)  Suppose that $\sr_{M/{\equiv}}([a]) = n < \infty$; we need to prove that $\sr_M(a) \le n$.  By Proposition \ref{srequiv.refmon}, it suffices to show that whenever $x,y \in M$ with $na+x = a+y$, we have $x \le y$.  Now $n[a]+[x] = [a]+[y]$, and since $\sr_{M/{\equiv}}([a]) = n$, there is some $e \in M$ such that $[e]+[x] = [y]$.  Consequently, $e+x \le y$ and thus $x \le y$, as desired.

(e) This is immediate from (c) and (d).
\end{proof}

\begin{proposition}  \label{M.mod.near-isom}
Let $S$ be a nonempty subset of $\Zpos$ such that $\Zpos \cdot S \subseteq S$, and let $\equiv$ be the congruence on $M$ defined by the rule
$$
u \equiv v \ \iff \ mu = mv\ \; \text{for some}\ \; m \in S.
$$

{\rm(a)} If $a \in M$ has finite stable rank, then $[a]_{\equiv}$ is Hermite, hence $\sr([a]_{\equiv}) \le 2$.

{\rm(b)} If $M$ is conical, then $M/{\equiv}$ is conical and $\sr([a]_{\equiv}) \le \sr(a)$ for all $a \in M$.
\end{proposition}

\begin{proof}
Abbreviate $[-]_{\equiv}$ to $[-]$.

(a) Let $a \in M$ with $\sr(a) = n < \infty$, and suppose that $2 [a] + [x] = [a] + [y]$ for some $x,y \in M$.  Then $m(2a+x) = m(a+y)$ for some $m \in S$, and so $mna + mna + mnx = mna + mny$.  Since $na \le mna + mnx$, Lemma \ref{W1.2} implies that $mna + mnx = mny$, and consequently $[a] + [x] = [y]$, because $mn \in S$.  This proves that $[a]$ is Hermite, and $\sr([a]) \le 2$ follows.

(b)  If $u,v \in M$ with $[u] + [v] = [0]$, then $mu+mv = 0$ for some $m \in S$.  Conicality of $M$ forces $u=v=0$, whence $[u] = [v] = [0]$, showing that $M/{\equiv}$ is conical.

It only remains to prove the final statement in case $\sr(a) = 1$, which under current hypotheses means that $a$ is cancellative.  If 
$[a] + [x] = [a] + [y]$ for some $x,y \in M$, then $m(a+x) = m(a+y)$ for some $m \in S$, whence $mx = my$ and thus $[x] = [y]$.  Therefore $[a]$ is cancellative and $\sr([a]) = 1$.
\end{proof}

\begin{proposition}  \label{M.mod.multi-isom}
Let $S$ be a nonempty subset of $\ZZ_{\ge2}$, and let $\equiv$ be the congruence on $M$ defined by the rule
$$
u \equiv v \ \iff \ mu = mv\ \; \text{for all}\ \; m \in S.
$$

{\rm(a)} Let $a \in M$ with finite stable rank, and set $n := \max(2,\sr(a))$.  If $n [a]_{\equiv} + [x]_{\equiv} = [a]_{\equiv} + [y]_{\equiv}$ for some $x,y \in M$, then $(n-1)[a]_{\equiv} + [x]_{\equiv} = [y]_{\equiv}$.  In particular, if $\sr(a) \le 2$, then $[a]_{\equiv}$ is Hermite.

{\rm(b)} $\sr([a]_{\equiv}) \le \max(2,\sr(a))$ for all $a \in M$. 

{\rm(c)} If $M$ is conical, then $M/{\equiv}$ is conical and $\sr([a]_{\equiv}) \le \sr(a)$ for all $a \in M$.
\end{proposition}

\begin{proof}
Abbreviate $[-]_{\equiv}$ to $[-]$.

(a) Let $a \in M$ with $n := \max(2,\sr(a)) < \infty$, and suppose that $n [a] + [x] = [a] + [y]$ for some $x,y \in M$.  Then $m(na+x) = m(a+y)$ for all $m \in S$.   For any $m \in S$, we have $m(n-1) \ge 2(n-1) \ge n$ because $n \ge 2$, whence $na \le m(n-1)a + mx$.  Since
$$
ma + m(n-1)a + mx = ma + my
$$
and $\sr(a) \le n$, Lemma \ref{W1.2} implies that $m(n-1)a + mx = my$.  Consequently, we obtain $(n-1)[a] + [x] = [y]$. 

(b) This follows from (a).

(c) The proof of Proposition \ref{M.mod.near-isom}(b) may be used, \emph{mutatis mutandis}.
\end{proof}

In case $S$ is an infinite subset of $\ZZ_{\ge2}$, Proposition \ref{M.mod.multi-isom} also holds with respect to the congruence $\equiv$ defined by the rule
$$
u \equiv v \ \iff \ mu = mv\ \; \text{for all sufficiently large}\ \; m \in S,
$$
with the same proof.

\sectionnew{Stable rank of multiples}  \label{sr.mult}

We continue to fix a commutative monoid $M$.

A famous theorem of Vaserstein \cite[Theorem 3]{Vas} established a formula for the stable rank of a matrix ring $M_k(S)$ in terms of $k$ and the stable rank of $S$.  It thus provides a formula for the stable rank of the endomorphism ring of a direct sum of $k$ copies of a module $A$ in terms of $k$ and the stable rank of the endomorphism ring of $A$.  Within the monoid of isomorphism classes of finitely generated projective modules over an exchange ring, this yields a formula for $\sr(k[A])$ in terms of $\sr([A])$.  We prove that this formula is valid in any refinement monoid (Theorem \ref{sr.ka.refine}), and that it holds up to an error of $1$ in any commutative monoid (Theorem \ref{W1.12}).  Many of the steps parallel Warfield's module-theoretic proof of Vaserstein's theorem \cite[Section 1]{War}.

We first observe that the stable ranks of positive multiples of any element of $M$ form a (non-strictly) decreasing sequence.

\begin{lemma}  \label{sr.decrease}
If $a \in M$ and $k,l \in \Zpos$ with $k \le l$, then $\sr(ka) \ge \sr(la)$.
\end{lemma}

\begin{proof}
If $\sr(ka) = \infty$, there is nothing to prove, so we assume that $\sr(ka) = n < \infty$.  It suffices to deal with the case when $l = k+1$.

Suppose $n(la) + x = la + y$ for some $x,y \in M$.  Add $(k-1)a$ to both sides of this equation and write the result as
$$
ka + (nk+n-1)a + x = ka + ka + y.
$$
Since $n(ka) \le (nk+n-1)a$, Lemma \ref{W1.2} implies that $nka + (n-1)a + x = ka + y$.  Since $\sr(ka) = n$, it follows that there is some $e \in M$ with $nka = ka + e$ and $e + (n-1)a + x = y$.  Adding $na$ to both sides of the penultimate equation and setting $e' := (n-1)a + e$, we obtain $n(la) = la + e'$ and $e' + x = y$, proving that $\sr(la) \le n$. 
\end{proof}

If an element $a \in M$ has finite stable rank, the decreasing sequence of stable ranks of multiples of $a$ eventually reaches $2$ or $1$, as follows.  

\begin{proposition}  \label{sr=n->na.hermite}
If $a \in M$ and $\sr(a) = n < \infty$, then $ka$ is Hermite for all $k \ge n$.  In particular, $ka$ is self-cancellative and $\sr(ka) \le 2$. 
\end{proposition}

\begin{proof}
Let $k \ge n$ be an integer, and suppose $2(ka)+x = ka+y$ for some $x,y \in M$.  Then 
$$
na + (2k-n)a + x = a + (k-1)a + y.
$$
Since $\sr(a) = n$, there is some $e \in M$ with $na = a+e$ and $e + (2k-n)a + x = (k-1)a + y$.  Observe that
$$
e + (2k-n)a + x = a+e + (2k-n-1)a+x = (2k-1)a + x,
$$
whence $(k-1)a + ka + x = (k-1)a + y$.  Since $na \le ka$, Lemma \ref{W1.2} implies that $ka + x = y$, proving that $ka$ is Hermite.  The remaining conclusions now follow from Proposition \ref{hermite.etc}(c).
\end{proof}

In Proposition \ref{sr=n->na.hermite}, we proved that if $a \in M$ and $\sr(a) = n < \infty$, then $na$ is Hermite.  Generally, smaller positive multiples of $a$ are not Hermite, or even self-cancellative.
 For instance, let $M$ and $a$ be as in Example \ref{exist.sra=?}(2).  Then $\sr(a) = n$, while $2(n-1)a = (n-1)a + b$ but $(n-1)a \ne b$, so that $(n-1)a$ is not self-cancellative.  Now fix an integer $k \ge 3$, and let $M_2,\dots,M_k$ be monoids as in Example \ref{exist.sra=?}(2) corresponding to $n=2,\dots,k$.  Each $M_n$ has a generator $a_n$ such that $\sr_{M_n}(a_n) = n$ while $(n-1)a_n$ is not self-cancellative.  Set $M := \prod_{n=2}^k M_n$ and $a := (a_2,\dots,a_k) \in M$.  Then $\sr_M(a) = k$, while $a,2a,\dots,(k-1)a$ all fail to be self-cancellative.

It is interesting to note that by the upcoming Theorem \ref{W1.12},  we will have $\sr((n-1)a) = 2$ for $a \in M$ with $\sr(a) = n \in \ZZ_{\ge2}$, even though $(n-1)a$ might not be self-cancellative.  The least positive multiple of $a$ that can possibly have stable rank $1$ is thus $na$.  By Proposition \ref{sr=n->na.hermite}, the stable ranks of the multiples of $a$ eventually stabilize to either $1$ or $2$ (and by Theorem \ref{W1.12} we will see that such stabilization does occur by $na$ at the latest).

\begin{lemma}
 \label{srb.squeeze}
Suppose that $a,b\in M$ satisfy  $a\le b \le na$ for some positive integer $n$. Then $\sr (b)\le \sr(a)$.
\end{lemma}

\begin{proof}
We may assume that $\sr(a) = k < \infty$.

 Write $b=a+p$ and $na= b+q$ for some $p,q\in M$. Observe that $a+p+q= na$.
Suppose that $kb+x = b+y$ for some $x,y\in M$. Adding $q$ to this equation, we get  \begin{equation}  \label{kbqx}
kb+q+x = na +y.
\end{equation}
 Now we have
 $$
 kb+q = ka +kp+q = ka+(p+q) + (k-1)p = ka+(n-1)a+ (k-1)p.
 $$
 Hence, \eqref{kbqx} can be written as
 $$
(n-1)a + [ka +(k-1)p+x] = (n-1)a + [a+y].
 $$
 Since $ka \le ka +(k-1)p+x$, we can use Lemma \ref{W1.2} to get
 $$
 ka + (k-1)p + x = a + y.
 $$
 Consequently, there is some $e \in M$ such that $ka= a+e$ and $e + (k-1)p+x= y$.
 Moreover observe that
 $$
 b+ [e+(k-1)p] = a+p+e+(k-1)p = (a+e)+ kp= ka+kp= kb.
 $$  
 Therefore $\sr (b)\le k = \sr (a)$, as desired.
 \end{proof}
 
 The case $b = na$ of Lemma \ref{srb.squeeze} says that $\sr(na) \le \sr(a)$ for all positive integers $n$, which also follows from either Theorem \ref{W1.9} or Lemma \ref{sr.decrease}.  We obtain a much tighter upper bound for $\sr(na)$ in Theorem \ref{W1.12}.

\begin{lemma}  \label{srma.sra}
If $a \in M$ and $m \in \Zpos$, then $\sr(a) \le m \cdot \sr(ma)$.
\end{lemma}

\begin{proof}
If $\sr(ma) = \infty$, the result holds with the usual convention that $m \cdot \infty = \infty$.

Now assume that $\sr(ma) = k < \infty$, and suppose that $kma + x = a + y$ for some $x,y \in M$.  Then $k(ma) + x + (m-1)a = (ma) + y$, whence there is some $e \in M$ such that $kma = ma + e$ and $e + x + (m-1)a = y$.  Setting $e' := (m-1)a + e$, we obtain $kma = a + e'$ and $e'+x = y$, verifying that $\sr(a) \le km$.
\end{proof}

\begin{theorem}  \label{srM(a)}
Let $a \in M$ and $n \in \Zpos$.

{\rm(a)} If $\sr(a) < \infty$, then all elements of $M(a)$ have finite stable rank.

{\rm(b)} $M(a)$ contains an element with stable rank $\le n$ if and only if $\sr(ka) \le n$ for some $k > 0$, if and only if $\sr(ka) \le n$ for $k \gg 0$.
\end{theorem}

\begin{proof}
(a) Given $b \in M(a)$, there exist $l,m \in \Zpos$ such that $b \le la$ and $a \le mb$.  Then $a \le mb \le lma$, and so Lemma \ref{srb.squeeze} implies that $\sr(mb) \le \sr(a) < \infty$.  Consequently, Lemma \ref{srma.sra} shows that $\sr(b) \le m \cdot \sr(mb) < \infty$.

(b) The reverse direction of the first equivalence holds \emph{a priori}, and the second equivalence follows from Lemma \ref{sr.decrease}.  It only remains to show that if there exists $b \in M(a)$ with $\sr(b) \le n$, then $\sr(ka) \le n$ for some $k>0$.  There exist $k,m \in \Zpos$ such that $b \le ka$ and $a \le mb$.  Since $b \le ka \le kmb$, Lemma \ref{srb.squeeze} yields $\sr(ka) \le \sr(b) \le n$, as desired.
\end{proof}

\begin{examples}  \label{arch.diff.sr}
In general, elements in the same archimedean component of $M$ need not have the same stable rank.  As we will see, knowing the finite stable rank of one element in an archimedean component does not generally limit the stable ranks of other elements.

(1) Let $M$ be as in Example \ref{exist.sra=?}(2) for some $n \ge 3$.  The generator $a \in M$ has stable rank $n$, and the archimedean component $M(a)$ also contains $na$, which has stable rank $2$ as follows.  Note first that since $a$ is not cancellative, neither is $na$, which implies $\sr(na) > 1$ because $M$ is conical.  On the other hand, $\sr(na) \le 2$ by Proposition \ref{sr=n->na.hermite}.

(2) This time, fix an integer $n \ge 2$, and let $M$ be presented with generators $a$ and $b$ and relations  $na+b = na$ and $2b = 0$.  The distinct elements of $M$ are $ma$ for $m \in \Znn$ together with $ma+b$ for $0< m< n$.  To see this, define a relation $\sim$ on $N := \Znn \times (\ZZ/2\ZZ)$ as follows:
$$
x \sim y \iff x= y\ \text{or}\ \exists\ k \in \ZZ_{\ge \mbox{$n$}}\ \text{and}\ p,q \in \ZZ/2\ZZ\ \text{such that}\ x = (k,p),\ y = (k,q).
$$
Then $\sim$ is a congruence on $N$, and $N/{\sim} \cong M$.

Now $M$ has two archimedean components, namely $M(b) = \{0,b\} = U(M)$ and $M(a) = M \setminus U(M)$. We claim that $\sr(a) = n$ while $\sr(na) = 1$.

First, consider $x,y \in M$ such that $na+x = na+y$.  Write $x = ma+c$ and $y = m'a+c'$ for some $m,m' \in \Znn$ and $c,c' \in U(M)$.  Then $na+x = (n+m)a$ and $na+y = (n+m')a$, from which we see that $m = m'$.  There is some $e \in U(M)$ such that $e+c = c'$.  Thus, $na = na+e$ and $e+x = y$, proving that $\sr(na) = 1$.

Next, note that $(n-1)a +(a+b) = a+ (n-1)a$, but there is no $e \in M$ satisfying $(n-1)a = a+e$ and $e+(a+b) = (n-1)a$.  Thus, $\sr(a) > n-1$.  On the other hand, $\sr(a) \le n\cdot\sr(na) = n$ by Lemma \ref{srma.sra}, and thus $\sr(a) = n$.
\end{examples}  

Working toward tight upper bounds for stable ranks of multiples, we adapt Warfield's proof of \cite[Theorem 1.11]{War} to a monoid form.  Given $a \in M$ and positive integers $k$, $l$, we consider the following condition, which might be called the \emph{$(k,l)$-stable rank condition for $a$}.
\begin{itemize}
\item[{$\sr_{k,l}[a]$}] If $ka+x = la+y$ for some $x,y \in M$, there exists $e \in M$ such that $ka = la+e$ and $e+x=y$.
\end{itemize}
Of course, $\sr_{k,1}[a]$ holds if and only if $\sr(a) \le k$.

\begin{theorem}  \label{W1.11}
Let $a \in M$ with $\sr(a) < \infty$.  There is a unique nonnegative integer $m_{a,M}$ such that for any integers $k \ge l \ge 1$, the condition $\sr_{k,l}[a]$ holds if and only if
$$
k \ge \sr(a) \quad \text{and} \quad k-l \ge m_{a,M} \,.
$$
Moreover, $m_{a,M} \le \sr(a) - 1$.
\end{theorem}

\begin{proof}  
First we show that if $\sr_{k,l}[a]$ holds for some $k \ge l \ge2$, then $\sr_{k,l-1}[a]$ also holds.  Suppose that $ka+x = (l-1)a+y$ for some $x,y \in M$.  Then $ka+(a+x) = la+y$.  From the hypothesis, there is some $e \in M$ with $ka = la+e$ and $e+(a+x) = y$.  Taking $e' := e+a$, then $ka = (l-1)a + e'$ and $e'+x = y$.

A consequence of the previous paragraph is that whenever $\sr_{k,l}[a]$ holds for some $k \ge l \ge1$, then $\sr_{k,1}[a]$ also holds.  This forces $k \ge \sr(a)$.

On the other hand, when $k \ge \sr(a)$, the condition $\sr_{k,1}[a]$ holds.  Consequently, for each integer $k \ge \sr(a)$, there is a greatest integer $l \in [1,k]$ such that $\sr_{k,l}[a]$ holds.  In other words, treating $k$ as fixed, but allowing $l$ to vary, there is a minimum value for $k-l$ where $\sr_{k,l}[a]$ holds.  The existence of $m_{a,M}$ is equivalent to the claim that this minimum value of $k-l$ remains stable as $k \ge \sr(a)$ varies.

To start verifying the claim, we next show that if $\sr_{k,l}[a]$ holds for $k\ge l\ge 1$ and $k \ge \sr(a)$, then $\sr_{k+1,l+1}[a]$ holds.  Suppose $(k+1)a + x = (l+1)a + y$ for some $x,y \in M$.  From Lemma \ref{W1.2}, $ka+x = la+y$.  By hypothesis, there exists some $e \in M$ with $ka = la+e$ and $e+x = y$.  Adding $a$ to the penultimate equality, we are done.

To finish the claim, we show that if $\sr_{k,l}[a]$ holds with $k \ge \sr(a)+1$ and $k \ge l \ge 2$, then $\sr_{k-1,l-1}[a]$ holds.  Suppose that $(k-1)a+x = (l-1)a+y$ for some $x,y \in M$.  Adding $a$ to both sides and using the hypothesis, there exists some $e \in M$ with $ka = la+e$ and $e+x = y$.  Applying Lemma \ref{W1.2} to the penultimate inequality yields $(k-1)a = (l-1)a+e$.  The claim is thus verified, proving the existence of $m_{a,M}$.  

Finally, note that since $\sr_{k,1}[a]$ holds with $k=\sr(a)$, we must have $m_{a,M} \le \sr(a) - 1$.
\end{proof}

 In what follows, $\lfloor\cdot\rfloor$ and $\lceil\cdot\rceil$ denote the standard floor and ceiling functions.  We will make use of the well-known identity
$$
\biggl\lceil \frac{n}{l} \biggr\rceil = 1 + \left\lfloor \frac{n-1}{l} \right\rfloor \qquad \forall\; n,l \in \ZZ,\; l>0
$$
(e.g., \cite[Ch.3, Exercise 12]{GKP}).

\begin{corollary}  \label{srla.via.maM}
Let $a \in M$ with $\sr(a) < \infty$, let $l \in \Zpos$, and let $m_{a,M}$ be as in Theorem {\rm\ref{W1.11}}.  Then $\sr(la)$ is the smallest positive integer $p$ such that $pl \ge \sr(a)$ and $(p-1)l \ge m_{a,M}$.  In other words,
$$
\sr(la) = \max \left( \left\lceil \frac{\sr(a)}{l} \right\rceil, \, 1+ \biggl\lceil \frac{m_{a,M}}{l} \biggr \rceil \right).
$$
\end{corollary}

\begin{proof}
Note that $\sr(la)$ is finite by, e.g., Theorem \ref{W1.9}.  
By definition of stable rank, $\sr(la)$ is the smallest positive integer $p$ such that $\sr_{pl,l}[a]$ holds.  The first statement of the theorem is thus immediate from Theorem \ref{W1.11}, and the second follows.
\end{proof}

We can now prove that an analog of Vaserstein's formula holds, up to an error of $1$, in any commutative monoid. 

\begin{theorem}  \label{W1.12}
If $a \in M$ with $\sr(a) = n < \infty$ and $l \in \Zpos$, then
\begin{equation}  \label{sr.ka}
1 + \left\lfloor \frac{n-1}{l} \right\rfloor \le \sr(la) \le 1 + \left\lceil \frac{n-1}{l} \right\rceil.
\end{equation}
In particular, if $l \mid n-1$ then $\sr(la) = 1 + \frac{n-1}{l}$.
\end{theorem}

\begin{proof}
Write $p := \sr(la)$.  By Corollary \ref{srla.via.maM},
$$
p \ge \biggl\lceil \frac{n}{l} \biggr\rceil = 1 + \left\lfloor \frac{n-1}{l} \right\rfloor.
$$
Now if $p' := 1 + \left\lceil \frac{n-1}{l} \right\rceil$, then $p'l \ge l+n-1 \ge n$ and $(p'-1)l \ge n-1 \ge m_{a,M}$, where the last inequality comes from Theorem \ref{W1.11}.  Corollary \ref{srla.via.maM} says that $p \le p'$, which provides the stated upper bound.

The final statement of the theorem follows immediately.
\end{proof}

The gap in \eqref{sr.ka} can be closed to an equality in case $M$ has refinement (Theorem \ref{sr.ka.refine}), but not in general.  On one hand, Example \ref{exist.sra=?}(3) contains an element $a$ such that $\sr(ma) = 2$ for all $m \in \Zpos$.  In particular, $\sr(2a) = 2 = 1 + \left\lceil \frac{2-1}{2} \right\rceil$.

On the other hand, in Example \ref{arch.diff.sr}(2), there is an element $a$ such that $\sr(a) = 2$ while $\sr(2a) = 1 = 1+ \left\lfloor \frac{2-1}{2} \right\rfloor$.  Such an example cannot be conical, since in the conical case, $\sr(2a) = 1$ would imply that $2a$ is cancellative (Proposition \ref{hermite.etc}(f)), whence $a$ is cancellative and $\sr(a) = 1$.

Conical examples where only the left hand inequality of \eqref{sr.ka} is an equality do exist, as follows.

\begin{example}  \label{arch.sr2&4}
Let $M$ be presented with generators $a$ and $b$ and relations $4a = 2a+b = 2b$.  We claim that the distinct elements of $M$ are $0$, $b$ and $a+b$ together with $na$ for $n \in \Zpos$.  

Obviously any element of $M$ has one of the given forms.  Let $M_0$ be the commutative monoid presented with generators $a_0$, $b_0$ and relations $3a_0 = a_0+b_0$ and $4a_0 = 2b_0$.  This is Example \ref{exist.sra=?}(2) with $n=3$.  As shown there, the elements $0,a_0, b_0, 2a_0,3a_0, \dots$ are all distinct.  Since there is a monoid homomorphism $M \rightarrow M_0$ sending $a \mapsto a_0$ and $b \mapsto b_0$, we find that the elements $0,a,b,2a,3a,\dots$ in $M$ are all distinct.  Now set
$$
S := \{ (m,n) \in \Znn^2 \mid m \ge 4\ \text{or}\ (m\ge2,\; n\ge 1)\ \text{or}\ n \ge 2 \},
$$
a semigroup ideal of $\Znn^2$.  Let $\sim$ be the congruence on $\Znn^2$ defined by $x \sim y \iff x = y\ \text{or}\ x,y \in S$, and set $M_1 := \Znn^2/{\sim}$.  Since there is a monoid homomorphism $M \rightarrow M_1$ sending $a \mapsto [(1,0)]_{\sim}$ and $b \mapsto [(0,1)]_{\sim}$, we see that $a+b$ is not equal to any of $0,a,b,2a,3a,\dots$, thus verifying the claim.

In particular, it follows that $M$ is conical.  

We claim that $\sr(a) = 4$ and $\sr(2a) = 2 = 1 + \left\lfloor \frac{4-1}{2} \right\rfloor$.

Since $a$ is not cancellative, neither is $2a$, and so $\sr(2a) \ne 1$ due to conicality of $M$.

Suppose that $2(2a)+x = (2a)+y$ for some $x,y \in M$.  Then $y \ne 0,a$.  If $y=b$, then $x=0$, and if $y=a+b$, then $x=a$.  In both of these cases, we can solve
\begin{equation}  \label{4a=2a+e}
4a = 2a+e \qquad \text{and} \qquad e+x = y
\end{equation}
with $e := b$.  Finally, if $y = ma$ for some integer $m \ge 2$, then we can solve \eqref{4a=2a+e} with $e := 2a$.

It now follows from Lemma \ref{srma.sra} that $\sr(a) \le 4$.  (One can also show this directly with an argument similar to that in the previous paragraph.)  On the other hand, $\sr (a) >3$ because we have the equation $3a+a= a+(a+b)$ but there is no $e\in M$ such that $3a=a+e$ and $a+e = a+b$.  Thus $\sr(a) = 4$.
\medskip

This example is universal in the following sense:  if $M'$ is a monoid containing an element $a'$ such that $\sr(a') = 4$ and $\sr(2a') = 2$, there must exist an element $b' \in M'$ such that
$$
4a' = 2a'+b' = 2b' \qquad \text{and} \qquad 3a' \ne a'+b'\,.
$$
Hence, there is a monoid homomorphism $M \rightarrow M'$ sending $a \mapsto a'$ and $b \mapsto b'$.

First, $\sr(a')\nleq 3$, so there exist some $x',y'\in M'$ with $3a'+x'=a'+y'$, and yet there is no element $e \in M'$ simultaneously satisfying $3a'=a'+e$ and $e+x'=y'$. Adding $a'$ to both sides of the starting equation, we get $4a'+x'=2a'+y'$.
Since $\sr(2a')=2$, there is some $b' \in M'$ such that $4a'=2a'+b'$ and $b'+x'=y'$. Due to the non-existence of an element $e$ as above, we must also have $3a' \neq a'+b'$.
Now, notice that $4a'+2a' = 4a'+b' =2a'+2b'$. So, since $\sr(2a')=2$, there is some $f\in M'$ with $4a'=2a'+f$ and $f+2a'=2b'$. In particular, $4a'=2b'$.

As we will see shortly (Corollary \ref{trichot}), a monoid $M'$ with the above properties cannot be separative.  It is interesting to see the non-separativity occur directly, since we have $2a'+2a'=2a'+b'=b'+b'$, but $2a' \neq b'$.
\end{example}

\begin{lemma}  \label{refmon.maM=sra-1}
Suppose that $M$ is a refinement monoid and $a \in M$ with $\sr(a) = n < \infty$.  Then the integer $m := m_{a,M}$ of Theorem {\rm\ref{W1.11}} equals $n-1$.
\end{lemma}

\begin{proof}
The proof of Theorem {\rm\ref{W1.11}} shows that $m = n-l$ where $l$ is the largest integer in $[1,n]$ such that $\sr_{n,l}[a]$ holds.  Since we are then done in case $n=1$, assume that $n \ge 2$.

We must show that $l=1$.  If $l \ge 2$, it follows as in the proof of Theorem {\rm\ref{W1.11}} that $\sr_{n,2}[a]$ holds.  Now consider $x,y \in M$ such that $(n-1)a+x = a+y$.  Then $na+x = 2a+y$ and $\sr_{n,2}[a]$ says that there is some $e \in M$ with $na = 2a+e$ and $e+x = y$, whence $x \le y$.  Proposition \ref{srequiv.refmon} thus implies that $\sr(a) \le n-1$, contradicting our hypotheses.

Therefore $l=1$, as required.
\end{proof}

\begin{theorem}  \label{sr.ka.refine}
Suppose that $M$ is a refinement monoid.
If $a \in M$ with $\sr(a) = n < \infty$ and $l \in \Zpos$, then
$$
\sr(la) = 1 + \left\lceil \frac{n-1}{l} \right\rceil.
$$
\end{theorem}

\begin{proof}
Set $p := \sr(la)$.  Then $p \le 1 + \left\lceil \frac{n-1}{l} \right\rceil$ by Theorem \ref{W1.12}.
The reverse inequality follows from Lemma \ref{refmon.maM=sra-1} and Corollary \ref{srla.via.maM}.
\end{proof}

\begin{proposition}  \label{sra.interval}
If $a \in M$ and $l,p \in \Zpos$ with $\sr(la) = p < \infty$, then
$$
lp-2l+2 \le \sr(a) \le lp.
$$
If $M$ has refinement, then $lp-2l+2 \le \sr(a) \le lp-l+1$.
\end{proposition}

\begin{proof}
We have $\sr(a) \le lp < \infty$ by Lemma \ref{srma.sra}.  Now set $n := \sr(a)$.  By Theorem \ref{W1.12},
$$
p \le 1 + \left\lceil \frac{n-1}{l} \right\rceil \le 1 + \frac{n-1}{l} + 1,
$$
whence $lp \le 2l+n-1$ and thus $n\ge lp-2l+1$.  Note that equality cannot hold here, since then $1 + \left\lceil \frac{n-1}{l} \right\rceil = p-1 < p$.  Therefore $n\ge lp-2l+2$.

If $M$ has refinement, then $p = 1 + \left\lceil \frac{n-1}{l} \right\rceil \ge 1 + \frac{n-1}{l}$ by Theorem \ref{sr.ka.refine}.  In this case, $lp \ge l+n-1$ and thus $n \le lp-l+1$.
\end{proof}

\sectionnew{Stable rank values in archimedean components}  \label{sr.value.arch}

We continue to fix a commutative monoid $M$.  We first investigate the set
$$
\sr_M(C) := \{ \sr_M(c) \mid c \in C \}
$$
of stable rank values within an archimedean component $C$ of $M$.  Namely, $\sr_M(C)$ must equal $\Zpos$, $\ZZ_{\ge2}$, $\{\infty\}$, or a finite subset of $\Zpos$.  When $M$ is conical, these possibilities are restricted a bit further.  Second, we carry over to monoids the trichotomy of \cite[Theorem 3.3]{AGOP} for stable ranks of finitely generated projective modules over separative exchange rings:  Assuming $M$ is separative, the only possible stable rank values for elements of $M$ are $1$, $2$, and $\infty$.  Additionally, when $M$ is separative, the function $\sr$ is constant on each archimedean component of $M$.

There is an obvious dichotomy within $M$: finite stable rank versus infinite stable rank.  This is respected by the archimedean components, as follows.

\begin{proposition}  \label{dichot.arch}
Let $C$ be an archimedean component of $M$.  Then either all elements of $C$ have finite stable rank, or all elements of $C$ have infinite stable rank.
\end{proposition}

\begin{proof}
Theorem \ref{srM(a)}(a).
\end{proof}

When $M$ is conical, we may refine this dichotomy to a trichotomy by separating out the elements of stable rank $1$.  That is not possible without conicality, as shown by Example \ref{arch.diff.sr}(2), in which there is an archimedean component containing elements with stable ranks $1$ and $2$.

\begin{proposition}  \label{trichot.arch}
Assume that $M$ is conical, and let $C$ be an archimedean component of $M$.  Then exactly one of the following occurs: $\sr_M(C) = \{1\}$, or $\sr_M(C) \subseteq \ZZ_{\ge2}$, or $\sr_M(C) = \{\infty\}$.
\end{proposition}

\begin{proof} In view of Proposition \ref{dichot.arch}, it suffices to prove that if $C$ contains an element $a$ with $\sr(a) = 1$, then $\sr(c) = 1$ for all $c \in C$.  Given $c \in C$, we have $b+c = na$ for some $b \in M$ and $n \in \Zpos$.  Due to conicality, $\sr(a) = 1$ implies that $a$ is cancellative, so it follows from $b+c = na$ that $c$ is cancellative, and thus $\sr(c) = 1$.
\end{proof}

We will improve Propositions \ref{dichot.arch} and \ref{trichot.arch} with the help of the following result.

\begin{proposition}  \label{lowvalue.srla}
Let $a \in M$ with $\sr(a) = n < \infty$.  If $k \in \ZZ_{\ge2}$ and 
$$
n \ge \max\bigl(2, k(k-1)(k-2)\bigr),
$$
 then there is some $l \in \Zpos$ such that $\sr(la) = k$.  
\end{proposition}

\begin{proof}
Let $m := m_{a,M}$ be as in Theorem \ref{W1.11}, and define $f_1,f_2 : \Zpos \rightarrow \Zpos$ by the rules
$$
f_1(l) := \biggl\lceil\frac{n}{l}\biggr\rceil \qquad \text{and} \qquad f_2(l) := 1 + \biggl\lceil\frac{m}{l}\biggr\rceil.
$$
Corollary \ref{srla.via.maM} says that $\sr(la) = \max(f_1(l),f_2(l))$ for all $l \in \Zpos$.  Note that $f_1$ and $f_2$ are non-increasing: if $l \ge l'$ in $\Zpos$, then $f_i(l) \le f_i(l')$.

\textbf{Observation 1}: If $l \in \Zpos$, then $f_1(l) = k$ if and only if $\frac{n}{k} \le l < \frac{n}{k-1}$.  Moreover, $f_2(l) = k$ if and only if $m>0$ and $\frac{m}{k-1} \le l < \frac{m}{k-2}$ (where we allow $\frac{m}0 = \infty$).

The first equivalence holds because $f_1(l) = k$ if and only if $k-1 < \frac{n}{l} \le k$.  The second is similar, taking into account that $f_2(l) = 1$ when $m=0$.

\textbf{Observation 2}: There exists $l' \in \Zpos$ such that $f_1(l') = k$.  In case $m>0$ and $m \ge (k-1)(k-2)$, there exists $l'' \in \Zpos$ such that $f_2(l'') = k$.

Since $n,k \ge 2$, we have $n \ge k(k-1)$.  Consequently, $\frac{n}{k-1} - \frac{n}{k} \ge 1$, and hence there is an integer $l'$ in the real interval $\left[ \frac{n}{k}, \frac{n}{k-1} \right)$.  By Observation 1, $f_1(l') = k$.  The second statement is proved in the same way.

We now split the proof into cases.  Assume first that $m=0$ or $m < (k-1)(k-2)$, and note that $k m < n$ in these cases.  Observation 2 provides a positive integer $l$ such that $f_1(l) = k$.  By Observation 1, $l \ge \frac{n}{k} > m$, from which it follows that $f_2(l) \le 2 \le k$.  Thus $\sr(la) = k$ in this case.

Finally, assume that $m>0$ and $m \ge (k-1)(k-2)$.  In this case, Observation 2 yields positive integers $l_1$ and $l_2$ such that $f_i(l_i) = k$ for $i=1,2$.  If $l_1 \ge l_2$, then $f_2(l_1) \le f_2(l_2) = k = f_1(l_1)$, whence $\sr(l_1a) = k$.  Likewise, if $l_1 \le l_2$, then $f_1(l_2) \le f_1(l_1) = k = f_2(l_2)$, and so $\sr(l_2a) = k$.
\end{proof}

Of course, under the hypotheses of Proposition \ref{lowvalue.srla}, there must be positive integers $l_2,l_3,\dots,l_k$ such that $\sr(l_j a) = j$ for $j = 2,3,\dots,k$.

\begin{theorem}  \label{srSinfinite}
Let $S$ be a subset of $M$ such that $\Zpos\cdot S \subseteq S$.  If the set $\sr_M(S)$ is infinite, then $\sr_M(S) \supseteq \ZZ_{\ge2}$.
\end{theorem}

\begin{proof}
Since $\sr_M(S)$ is infinite, it must contain infinitely many positive integers.
Given $k \in \ZZ_{\ge2}$, choose an element $a \in S$ with $\sr(a) = n < \infty$ and $n \ge k^3$.  By Proposition \ref{lowvalue.srla}, there exists $l \in \Zpos$ such that $\sr(la) = k$, and $la \in S$ by hypothesis.
\end{proof}

Propositions \ref{dichot.arch} and \ref{trichot.arch} can now be upgraded via Theorem \ref{srSinfinite}:

\begin{theorem}  \label{srC}
Let $C$ be an archimedean component of $M$.

{\rm(a)} The set $\sr_M(C)$ equals one of $\Zpos$, $\ZZ_{\ge2}$, $\{\infty\}$, or a finite subset of $\Zpos$.

{\rm(b)} If $M$ is conical, then $\sr_M(C)$ equals one of $\{1\}$, $\ZZ_{\ge2}$, $\{\infty\}$, or a finite subset of $\ZZ_{\ge2}$.
\qed\end{theorem}

\begin{examples}  \label{examples.srC}
For any monoid $M$, the group $U(M)$ is an archimedean component of $M$, and $\sr_M(U(M)) = \{1\}$.  We thus restrict attention to archimedean components not containing any units.

(1) Let $M := \Znn \sqcup \{\infty\}$ as in Example \ref{exist.sra=?}(1).  Then $\sr_M(\Zpos) = \{1\}$ and $\sr_M(\{\infty\}) = \{\infty\}$.

(2) Let $M$ be as in Example \ref{exist.sra=?}(3).  Then $\sr_M(\Zpos b) = \{1\}$ (since $b$ is cancellative) and $\sr_M(\Zpos a) = \{2\}$.

(3) Let $M$ be as in Example \ref{exist.sra=?}(2), with $n \ge3$.  As shown in the example, $\sr(a) = n$, whence $n \ge \sr(ma) \ge 2$ for all $m \in \Zpos$, since $ma$ is not cancellative.  As also shown, $\sr(b) = 2$.  Therefore $\sr_M(M \setminus \{0\}) \subseteq [2,n]$.

The set $\sr_M(M \setminus \{0\})$ need not equal the full integer interval $[2,n]$, however.  For instance, if $n=5$, then $\sr_M(M \setminus \{0\}) = \{2,3,5\}$.  Namely, from $\sr(a) = 5$, we get $\sr(2a) = 3$ and $\sr(4a) = 2$ by Theorem \ref{W1.12}, and then $\sr(ma) \le 3$ for all $m>2$ by Lemma \ref{sr.decrease}.   

Similarly, if $n=7$, then $\sr_M(M \setminus \{0\}) = \{2,3,4,7\}$.

(4) Let $M$ be as in Example \ref{arch.diff.sr}(2), with $n \ge 3$.  As shown there, $\sr(a) = n$ and $\sr(na) = 1$.  If $z = ma+b$ with $0<m<n$, then $\sr(z) = \sr(ma) \le n$ by Corollary \ref{sr+unit} and Theorem \ref{W1.9}.  As in Example (3), we conclude that $\sr_M(M \setminus U(M)) \subseteq [1,n]$.  Similarly, if $n=5$, then $\sr_M(M \setminus U(M)) = \{1,2,3,5\}$, while if $n=7$, then $\sr_M(M \setminus U(M)) = \{1,2,3,4,7\}$.

(5) In Theorem \ref{simple.crm.sr[2,infty)}, we will establish the existence of simple conical refinement monoids $N$ with $\sr_N(N\setminus\{0\}) = \ZZ_{\ge2}$.
\end{examples}

We now turn to the influence of separativity on stable ranks.
The proof of \cite[Theorem 3.3]{AGOP} converts directly to the monoid setting and yields the following two results.

\begin{proposition}  \label{sep.srfinite.le2}
Assume that $M$ is separative.  Then an element $a \in M$ has finite stable rank if and only if $\sr(a) \le 2$, if and only if $a$ is Hermite.
\end{proposition}

\begin{proof}
It suffices to show that $\sr(a) = n < \infty$ implies that $a$ is Hermite.  Given $x,y \in M$ such that $2a+x = a+y$, we want to show that $a+x = y$.  In view of Lemma \ref{agop2.1}(d), it suffices to show that $a \in \langle y \rangle$.  Add $(n-1)y$ to both sides of $2a+x = a+y$ to get $2a + x + (n-1)y = a + ny$.  By repeatedly replacing $a+y$ by $2a+x$ on the left hand side, we find that
$$
na + (a+nx) = a+ny.
$$
Since $\sr(a) = n$, there exists $e \in M$ such that $na = a+e$ and $e+(a+nx) = ny$.  The latter equation shows that $a \le ny$, and so $a \in \langle y \rangle$ as required.
\end{proof}

Separative monoids thus satisfy the following elementwise trichotomy:

\begin{corollary}  \label{trichot}
 If $M$ is separative, then every element of $M$ has stable rank $1$, $2$, or $\infty$.
\end{corollary}

The monoid $M := \Znn \sqcup \{\infty\}$ of Example \ref{exist.sra=?}(1) is easily seen to be separative, and $1$, $\infty$ both occur as stable ranks of elements of $M$.  An example of a separative monoid containing elements of stable rank $2$ follows.

\begin{example}  \label{sep.with.sr2}
Let $M$ be the monoid of Example \ref{exist.sra=?}(3), and observe that $M$ is conical.  By direct calculation, one can show that $M$ is separative, as well as having refinement (which we shall need later).  Alternatively, $M$ is a \emph{graph monoid} in the sense of \cite[item (M), p.163]{AMP}, corresponding to the directed graph
$$
\xymatrixcolsep{5ex}  \xymatrix{
a \ar@(dl,ul) \ar[r] &b
}$$
and so \cite[Theorem 6.3 and Proposition 4.4]{AMP} imply that $M$ is a separative refinement monoid.  

As shown in Example \ref{exist.sra=?}(3), the generator $a \in M$ has the property that $\sr(ma) = 2$ for all $m \in \Zpos$.
\end{example}

\begin{lemma}  \label{ref+hermite}
Assume that $M$ is a refinement monoid.  Then $a \in M$ is Hermite if and only if whenever $2a+x = a+y$ for some $x,y \in \langle a \rangle$, then $a+x = y$.
\end{lemma}

\begin{proof}
It suffices to establish the reverse implication.  Suppose $2a+x = a+y$ for some arbitrary $x,y \in M$.  Use refinement to write $2a = a_1+a_2$ and $x = x_1+x_2$ with $a_1+x_1 = a$ and $a_2+x_2 = y$.  Then we have
$$
2a+x_1 = a_1+a_2+x_1 = a+a_2\,,
$$
with $x_1 \le a$ and $a_2 \le 2a$, so that $x_1,a_2 \in \langle a \rangle$.  Therefore, we get $a+x_1 = a_2$, and so $a+x = a+x_1+x_2 = a_2+x_2 = y$.
\end{proof}

We can now state:

\begin{theorem}  \label{sep.trichot.arch}
Given $a \in M$, consider the following conditions:
\begin{enumerate}
\item $\sr(a) < \infty$.
\item $\sr(a) \le 2$.
\item $a$ is Hermite.
\item All elements of $M(a)$ are Hermite.
\item All elements of $M(a)$ are self-cancellative.
\end{enumerate}

If $M$ is separative, then \textup{(1)--(4)} are equivalent.
Moreover, all the elements in any archimedean component of $M$ have the same stable rank.

If $M$ has refinement, then $(4)$ and $(5)$ are equivalent.
\end{theorem}

\begin{proof}
The implications (4)$\Longrightarrow$(3)$\Longrightarrow$(2)$\Longrightarrow$(1) and (4)$\Longrightarrow$(5) hold without any assumptions on $M$, and follow from Proposition \ref{hermite.etc}.

Asuming $M$ is separative, we show (1)$\Longrightarrow$(4).  By Theorem \ref{srM(a)}(a), condition (1) implies that all elements of $M(a)$ have finite stable rank, and so by Proposition \ref{sep.srfinite.le2} they are Hermite.

We next address the statement about archimedean components.  Given the equivalence above and the trichotomy of Corollary \ref{trichot}, all that remains is to show that if $a \in M$ with $\sr(a) = 1$ and $b \in M(a)$, then $\sr(b) = 1$.  Then $b \le ma$ and $a \le nb$ for some positive integers $m$ and $n$.  Since $a \le nb \le mna$, Lemma \ref{srb.squeeze} implies that $\sr(nb) = 1$.  Now suppose that $b+x = b+y$ for some $x,y \in M$.  Then $nb+x = nb+y$, and so there exists $e \in M$ such that $nb = nb+e$ and $e+x = y$.  Applying Lemma \ref{agop2.1}(d), we obtain $b = b+e$, proving that $\sr(b) = 1$.

Finally, assuming that $M$ has refinement (but no longer assuming separativity), we show (5)$\Longrightarrow$(4).  Since $M(b) = M(a)$ for all $b \in M(a)$, it suffices to show that (5) implies $a$ is Hermite.  We use the specialized characterization of the Hermite property in Lemma \ref{ref+hermite}.  Suppose that $2a+x = a+y$ for some $x,y \in \langle a \rangle$.  Then we have
$$
2(a+x) = (a+x)+y,
$$
and $a+x \in M(a)$, so that we get $a+x = y$ by (5), showing that $a$ is Hermite.
\end{proof}

We note that one still cannot add ``$a$ is self-cancellative" to the list of equivalent conditions in Theorem \ref{sep.trichot.arch}, even assuming the monoid is conical as well as having refinement and being separative. Indeed, the three element monoid $M = \{0,a,2a\}$ of Example \ref{sr2.not.hermite}(2) again provides a counterexample. (The facts that $M$ is separative and has refinement are straightforward.)

As noted in Proposition \ref{hermite.etc}, with no extra hypotheses on $M$, if $M$ is strongly separative (i.e., all its elements are self-cancellative), then $M$ is Hermite (i.e., all its elements are Hermite), and conversely.  This was also previously given as the equivalence (a)$\Longleftrightarrow$(d) of Lemma \ref{agop.p126}.

\sectionnew{Stable rank values in refinement monoids}  \label{sr.refmon}

We continue to investigate sets of stable rank values, concentrating now on refinement monoids.  As before, we continue to fix a commutative monoid $M$.

When $M$ is a simple conical refinement monoid, we prove that $\sr_M(M \setminus \{0\})$ must be one of $\{1\}$, $\ZZ_{\ge2}$, or $\{\infty\}$.  In particular, if $\sr_M(M)$ is bounded above, then $M$ is cancellative.  Dropping simplicity, when $M$ is a conical refinement monoid and $C$ is an archimedean component of $M$, we prove that $\sr_M(C)$ must equal one of $\{1\}$, $\ZZ_{\ge2}$, $\{\infty\}$, or a finite subset of $\ZZ_{\ge2}$.

Recall that an element $a \in M$ is \emph{irreducible} provided $a$ is not a unit and whenever $a = b+c$ in $M$, either $b$ or $c$ is a unit.  If $M$ is conical, this just means one of $b$ or $c$ equals $0$ and the other equals $a$.

The following facts are well known; see, e.g., \cite[Lemma 2.1, Proposition 2.2]{AGtamewild}.

\begin{lemma}  \label{atom.in.crm}
Let $M$ be a conical refinement monoid.

{\rm(a)} Let $a \in M$ be an irreducible element.  Then $a$ is cancellative, $\langle a \rangle = \{ ma \mid m \in \Znn\}$, and all elements of $\langle a \rangle$ are cancellative in $M$.  In particular, $\Znn \cong \langle a \rangle$ via $m \mapsto ma$.

{\rm(b)} The submonoid of $M$ generated by all irreducible elements is an o-ideal of $M$, and all elements of this submonoid are cancellative in $M$.
\end{lemma}

\begin{lemma}  \label{simple.crm.noatoms}
Let $M$ be a simple conical refinement monoid that has no irreducible elements.  For any nonzero $a_1,\dots,a_k \in M$ and $n \in \Zpos$, there exists a nonzero $c \in M$ such that $n c \le a_i$ for all $i \in [1,k]$.
\end{lemma}

\begin{proof}  \cite[Lemma 5.1(a)]{Par}.
\end{proof}

Analogs of the following theorem were first proved for (monoids of) finitely generated projective modules over stably finite simple C*-algebras (\cite[Theorem 2.2]{Rie} with \cite [Theorem A1]{Bla}) and then for (monoids of) finitely generated projective modules over simple regular rings \cite[Theorem 1.2]{AGalmost}.

\begin{theorem}  \label{simple.crm.bdd.sr}
Let $M$ be a simple conical refinement monoid.  If there exists $n \in \Zpos$ such that all elements of $M$ have stable rank at most $n$, then $M$ is cancellative.
\end{theorem}

\begin{proof}
If there is an irreducible element $a \in M$, then by Lemma \ref{atom.in.crm}, all elements of $\langle a \rangle$ are cancellative in $M$.  Since $\langle a \rangle = M$ by simplicity, we are done in this case.

Assume now that $M$ has no irreducible elements, and consider $a,x,y \in M$ with $a+x = a+y$.  Since there is nothing to do if $x = y = 0$, we may assume that $x \ne 0$.  By Lemma \ref{simple.crm.noatoms}, there is a nonzero $z \in M$ such that $nz \le x$.  Using simplicity again, $a+u = mz$ for some $u \in M$ and $m \in \Zpos$.  Adding $u$ to both sides of $a+x = a+y$, we obtain $mz+x = mz+y$.  Since $\sr(z) \le n$ and $nz \le x$, Lemma \ref{W1.2} implies that $x=y$.
\end{proof}

A trichotomy for the stable rank values in a simple conical refinement monoid follows:

\begin{theorem}  \label{trichot.simple.crm}
Let $M$ be a simple conical refinement monoid.  Then exactly one of the following holds:

{\rm(a)} $\sr(a) = 1$ for all $a \in M$.

{\rm(b)} $\sr(a) = \infty$ for all nonzero $a \in M$.

{\rm(c)} The set $\sr_M(M \setminus \{0\})$ equals $\ZZ_{\ge2}$.
\end{theorem}

\begin{proof}
These three conditions are obviously pairwise incompatible.  Assume that $M$ does not satisfy (a) or (b); we must prove that (c) holds.  Since $M \setminus \{0\}$ is an archimedean component of $M$ (by simplicity), Proposition \ref{trichot.arch} implies that $\sr_M(M \setminus \{0\}) \subseteq \ZZ_{\ge2}$.  There is no finite upper bound on $\sr_M(M \setminus \{0\})$, in view of Theorem \ref{simple.crm.bdd.sr}.
Therefore we conclude from Theorem \ref{srSinfinite} that $\sr_M(M \setminus \{0\}) = \ZZ_{\ge2}$.
\end{proof} 

It is easy to find simple conical refinement monoids for which cases (a) or (b) of Theorem \ref{trichot.simple.crm} hold.  E.g., for (a) take $M = \Znn$, and for (b) let $M$ be the $2$-element monoid $\{0,\infty\}$.  We will construct examples of case (c) in the following section (see Theorem \ref{simple.crm.sr[2,infty)}).

When $M$ is simple and conical, it has two archimedean components, namely $\{0\}$ and $M \setminus \{0\}$.  Assuming $M$ also has refinement, Theorem \ref{trichot.simple.crm} implies that for each archimedean component $C$ of $M$, the set $\sr_M(C)$ is equal to one of $\{1\}$, $\{\infty\}$, or $\ZZ_{\ge2}$.  In the non-simple (but conical) case, Theorem \ref{srC}(b) says that the only other possibilities are finite subsets of $\ZZ_{\ge2}$.  One such, at least, is known:

\begin{example}  \label{arch.comp.sr2}
Let $M$ be the monoid of Examples \ref{exist.sra=?}(3) and \ref{sep.with.sr2}; then $M$ is a conical refinement monoid (and also separative).  Now $M$ has three archimedean components: $\{0\}$, $\Zpos\,b$, and $\Zpos\,a$.  As noted in Example \ref{exist.sra=?}(3), $\sr(ma) = 2$ for all $m \in \Zpos$, and thus
$$
\sr_M(\Zpos\,a) = \{2\}.
$$
\end{example}

No examples are known of a conical refinement monoid $M$ that has an archimedean component $C$ such that $\sr_M(C)$ is a finite subset of $\ZZ_{\ge2}$ other than $\{2\}$.  That possibility can be ruled out when $C$ satisfies certain weak divisibility conditions, as we will show in Theorem \ref{sr.C=C+C}.

\begin{lemma}  \label{arch.sr.down}
Let $C$ be an archimedean component in a monoid $M$.  Then for any $a,b \in C$ with $a \le b$, it follows that $\sr(a) \ge \sr(b)$.
\end{lemma}

\begin{proof}  Since $b$ is in the same component as $a$, we have $b \le na$ for some positive integer $n$.  Apply Lemma \ref{srb.squeeze}.
\end{proof}

\begin{lemma}  \label{arch.down}
{\rm\cite[Lemma 2.2]{Wembed}}
If $M$ is a refinement monoid, then every archimedean component of $M$ is downward directed.
\end{lemma}

\begin{theorem}  \label{sr.C=C+C}
Let $C$ be an archimedean component in a conical refinement monoid $M$, and assume that $C = C + C$ {\rm(}i.e., every element of $C$ is a sum of two elements from $C${\rm)}.  Then $\sr_M(C)$ equals one of $\{1\}$, $\{2\}$, $\ZZ_{\ge2}$, $\{\infty\}$.
\end{theorem}

\begin{proof}
In view of Theorem \ref{srC}(b), we only need to consider the case when $\sr_M(C)$ is a finite subset of $\ZZ_{\ge2}$.  

Set $n := \max \sr_M(C)$ and choose $c \in C$ with $\sr(c) = n$.  In view of Lemma \ref{arch.sr.down}, it follows that $\sr(a) = n$ for all $a \in C$ with $a \le c$.  By hypothesis, $c = c_1+c_2$ for some $c_1,c_2 \in C$.  Lemma \ref{arch.down} shows that there is some $b \in C$ with $b \le c_i$ for $i=1,2$, whence $2b \le c$.  Consequently, $\sr(b) = \sr(2b) = n$.  By Theorem \ref{sr.ka.refine}, $\sr(2b)$ equals $(n+1)/2$ if $n$ is odd, or $(n+2)/2$ if $n$ is even.  But this means $n \le (n+2)/2$, which is impossible unless $n=2$.  

Therefore $\sr_M(C) = \{2\}$ in this case.
\end{proof}

\sectionnew{Stable ranks ranging over $\ZZ_{\ge2}$}  \label{sr.Zge2}

This section is devoted to the construction of simple conical refinement monoids for which the stable ranks of the nonzero elements range over all of $\ZZ_{\ge2}$.

\begin{definition}  \label{def.unitary}
A monoid homomorphism $\phi : M \rightarrow N$ is \emph{weakly unitary} if
\begin{itemize}
\item Whenever $u,v \in M$ and $z \in N$ with $\phi(u) + z = \phi(v)$, then $z \in \phi(M)$.
\end{itemize}
Following \cite[Definition 1.2]{Wembed}, we say that $\phi$ is \emph{unitary} if
\begin{itemize}
\item $\phi$ is injective;
\item $\phi(M)$ is cofinal in $N$ with respect to the algebraic ordering;
\item $\phi$ is weakly unitary.
\end{itemize}
\end{definition}

We will make crucial use of one of Wehrung's embedding theorems \cite[Corollary 2.7]{Wembed}, which states that any simple conical monoid can be unitarily embedded in a simple conical refinement monoid.  

\begin{lemma}  \label{unitary.sr.incr}
Let $\phi : M \rightarrow N$ be a unitary monoid homomorphism.  Then $\sr_N(\phi(a)) \ge \sr_M(a)$ for all $a \in M$.
\end{lemma}

\begin{proof}
It suffices to prove the following: If $n \in \Zpos$ and $\sr_N(\phi(a)) \le n$, then $\sr_M(a) \le n$.

Assume $n \in \Zpos$ with $\sr_N(\phi(a)) \le n$, and let $x,y \in M$ such that $na+x = a+y$.  Since $n \phi(a) + \phi(x) = \phi(a) + \phi(y)$, there exists $f \in N$ such that $n\phi(a) = \phi(a) + f$ and $f + \phi(x) = \phi(y)$.  Unitarity of $\phi$ implies that $f = \phi(e)$ for some $e \in M$, and then $na = a+e$ and $e+x = y$ because $\phi$ is injective.  Therefore $\sr_M(a) \le n$.
\end{proof}

We aim to construct unitary embeddings of simple conical monoids into refinement monoids that nearly preserve stable ranks, at least up to an error of $1$.  Most of the work is done in terms of the following higher-level analogs of the Hermite condition.

\begin{definition}  \label{def.sr+}
Let $a \in M$ and $m \in \Zpos$.  Let us say that $a$ satisfies the \emph{strong $m$-stable rank condition in $M$} provided that whenever $ma+x = a+y$ for some $x,y \in M$, it follows that $(m-1)a + x = y$.  Clearly the strong $m$-stable rank condition implies the strong $m'$-stable rank condition for all $m' > m$.  Define the \emph{strong stable rank of $a$ in $M$}, denoted $\sr^+_M(a)$, to be the smallest positive integer $m$ such that $a$ satisfies the strong $m$-stable rank condition in $M$ (if such $m$ exist) or $\infty$ (otherwise).

The strong $m$-stable rank condition implies the $m$-stable rank condition, and Lemma \ref{W1.2} shows that the $m$-stable rank condition implies the strong $(m+1)$-stable rank condition.  Consequently,
\begin{equation}  \label{sr.vs.sr+}
\sr_M(a) \le \sr^+_M(a) \le \sr_M(a)+1 \qquad \forall\; a \in M.
\end{equation}
\end{definition}

We have implicitly used strong stable rank arguments previously, such as in Proposition \ref{max.antisymm.quo}(b) (which explains the $+1$ term there), as well as in Proposition \ref{sr=n->na.hermite}.  The interested reader is welcome to prove an analog of Theorem \ref{W1.12} for strong stable ranks.

Clearly Lemma \ref{unitary.sr.incr} also holds for strong stable rank.  In fact:

\begin{observation}  \label{injec.sr+.incr}
If  $\phi : M \rightarrow N$ is an injective monoid homomorphism, then $\sr^+_N(\phi(a)) \ge \sr^+_M(a)$ for all $a \in M$.
\end{observation}

Given a congruence $\equiv$ on a monoid $N$, let $\pi_\equiv : N \rightarrow N/{\equiv}$ denote the quotient map.

\begin{lemma}  \label{lem1}
Let $N$ be a monoid, $b \in N$, and $m \in \Zpos$.

{\rm(a)} There is a smallest congruence $\equiv$ on $N$ such that $\sr^+_{N/{\equiv}}(\pi_{\equiv}(b)) \le m$.

{\rm(b)} If $u,v \in N$ satisfy $u \equiv v$, then either $u=v$ or there exist a positive integer $k$ and a sequence of elements $u_0=u,u_1,\dots, u_t=v$ in $N$ such that for each $j\in [1,t]$, one of $(u_{j- 1},u_j)$ or $(u_j,u_{j-1})$ equals $((m-1)b+x_j,y_j)$ for some $x_j,y_j \in N$ with $(k+m-1)b + x_j = kb + y_j$.

{\rm(c)} If $u,v \in N$ satisfy $u \equiv v$, then there is some $w \in \Znn\,b$ such that $w+u = w+v$.
\end{lemma}

\begin{proof}
(a) The stated property is a universally quantified implication for any congruence $\approx$, since $\sr^+_{N/{\approx}}(\pi_{\approx}(b)) \le m$ if and only if
\begin{equation}  \label{sr+mod.cong}
(mb+x \approx b+y) \implies ((m-1)b+x \approx y), \qquad \forall\; x,y \in N.
\end{equation}
Thus, any intersection of congruences satisfying \eqref{sr+mod.cong} will satisfy \eqref{sr+mod.cong}.
Take $\equiv$ to be the intersection of the set of all congruences on $N$ satisfying \eqref{sr+mod.cong} (which is nonempty because it contains $N\times N$). 

(c) This will follow from (b) by taking $w := 0$ or $w := kb$ in the respective cases.

(b) It is convenient to set $n:= m-1\in \ZZ_{\geq 0}$ and
$$
X:=\{(x,y)\in N\times N \mid (n+k)b+x = kb+y \text{\ for some\ } k\in \ZZ_{>0}\}. 
$$
Observe that $X$ is closed under $N$-translations, in the sense that $(x,y)\in X$ implies that $(x+z,y+z)\in X$ for any $z\in N$. Let 
$$
\sim\ := (nb,0)+X = \{ (nb+x,y) \mid (x,y) \in X \},
$$
and let $\approx$ be the smallest equivalence relation containing $\sim$. Notice that $\approx$ is a congruence, due to closure under $N$-translations. Since $\sim$ is contained in $\equiv$, so is $\approx$.

We now show that $\approx$ satisfies the property stated in (b) for $\equiv$.  Suppose that $u,v\in N$ with $u\approx v$ but $u \ne v$.  Then there is a sequence of elements $u_0=u,u_1,\ldots, u_t=v$ in $N$ such that for each $j\in [1,t]$ we have $(u_{j- 1},u_j)$ in $\sim$ or $\sim^{-1}$. For each such $j$, there exists $(x_j,y_j) \in X$ such that $(u_{j- 1},u_j)$ equals $(nb+x_j, y_j)$ or $(y_j, nb+x_j)$.  By definition of $X$, there are some $k_j\in \ZZ_{>0}$ for which $(n+k_j) b + x_j = k_jb + y_j$, whence $k_jb+ u_{j-1} = k_jb +u_j$. Thus, taking $k$ to be the maximum of these $k_j$, we see that $(k+m-1)b + x_j = kb + y_j$ for $j \in [1,t]$.  Moreover, $kb+ u_{j-1} = kb +u_j$ for $j \in [1,t]$, whence $kb+u = kb+v$.

This verifies the condition of (b) for $\approx$.  Part (b) itself will follow once we show that $\approx$ is the same as $\equiv$.

To that end, it suffices to show that $\pi_{\approx}(b)$ satisfies the strong $m$-stable rank condition in $N/{\approx}$. Suppose that $m\pi_{\approx}(b) + x = \pi_{\approx}(b) + y$, for some $x,y\in N/{\approx}$. Fix $x',y'\in N$ with $\pi_{\approx}(x')=x$ and $\pi_{\approx}(y')=y$. Thus, $mb+x'\approx b+y'$.  From the previous paragraph, we know that there exists some positive integer $k$ such that $(n+k)b+x' = kb+y'$.  Hence, $nb+x' \sim y'$ and therefore $n\pi_{\approx}(b)+x=y$.
\end{proof}

\begin{lemma}  \label{lem2}
Let $N$ be a monoid, $(b_k)_{k\in K}$ a nonempty family of elements of $N$, and $(m_k)_{k\in K}$ a corresponding family of positive integers.

{\rm(a)} There is a smallest congruence $\equiv$ on $N$ such that $\sr^+_{N/{\equiv}}(\pi_{\equiv}(b_k)) \le m_k$ for all $k \in K$.

{\rm(b)} If $u,v \in N$ satisfy $u \equiv v$, then there is some $w$ in the submonoid $\sum_{k\in K} \Znn\,b_k$ such that $w+u = w+v$.
\end{lemma}

\begin{proof}
(a) This holds for the same reason as Lemma \ref{lem1}(a).

(b) We induct on $\kappa := \card(K)$, the case $\kappa = 1$ being Lemma \ref{lem1}(c).  Now let $\kappa > 1$ and assume that part (b) holds for congruences obtained as in (a) from nonempty families indexed by sets with cardinality less than $\kappa$.  We break the induction step into the cases when $\kappa$ is finite or infinite.

($\kappa$ finite).  Identify $K$ with $[1,\ell]$ for some integer $\ell \ge 2$.  We define a countable sequence of congruences
$$
\approx_0\ \subseteq\ \approx_1\ \subseteq\ \approx_2\ \subseteq\ \cdots\ \subseteq\ \approx
$$
on $N$ as follows.  First, let $\approx_0$ be the equality relation.  Supposing that $\approx_j$ has been defined for some even nonnegative integer $j$, let $\approx_{j+1}$ be the smallest congruence containing $\approx_j$ such that $\sr^+_{N/{\approx_{j+1}}}(\pi_{\approx_{j+1}}(b_k)) \le m_k$ for all $k \in [1,\ell-1]$.  (This congruence exists by applying the inductive hypothesis to $N/{\approx_j}$ coupled with the correspondence theorem for congruences on $N/{\approx_j}$ and $N$.)  Once $\approx_j$ has been defined for some odd positive integer $j$,  let $\approx_{j+1}$ be the smallest congruence containing $\approx_j$ such that $\sr^+_{N/{\approx_{j+1}}}(\pi_{\approx_{j+1}}(b_\ell)) \le m_\ell$.  Finally, let $\approx$ be the union of the chain of $\approx_j$.

An easy inductive argument shows that all $\approx_j$ are contained in $\equiv$, whence $\approx\ \subseteq\ \equiv$.  On the other hand, it is straightforward to show that $\sr^+_{N/{\approx}}(\pi_{\approx}(b_k)) \le m_k$ for all $k \in [1,\ell]$.  Thus, $\approx$ equals $\equiv$.

Now consider $u,v \in N$ satisfying $u \equiv v$.  Then $u \approx_j v$ for some $j \ge 0$, and we proceed by a secondary induction on $j$.  If $j=0$, then $u = v$, so we assume that $j>0$.  Writing $i := j-1$ and letting $\equiv_j$ denote the congruence on $N/{\approx_i}$ induced from $\approx_j$, we have $\pi_{\approx_i}(u) \equiv_j \pi_{\approx_i}(v)$.  If $i$ is even, it follows from the construction of $\approx_j$ and the inductive hypothesis that $w+\pi_{\approx_i}(u) = w+\pi_{\approx_i}(v)$ for some $w \in \sum_{k=1}^{\ell-1} \Znn\,\pi_{\approx_i}(b_k)$.  Similarly, if $i$ is odd, it follows from the construction of $\approx_j$ and Lemma \ref{lem1}(b) that $w+\pi_{\approx_i}(u) = w+\pi_{\approx_i}(v)$ for some $w \in \Znn\,\pi_{\approx_i}(b_\ell)$.  In either case, we obtain an element $w_i \in \sum_{k\in K} \Znn\,b_k$ such that $w_i + u \approx_i w_i + v$.  By our secondary induction, there is some $w \in \sum_{k\in K} \Znn\,b_k$ such that $w+w_i+u = w+w_i+v$.  This completes the secondary induction, and the proof of (b), for $K = [1,\ell]$.

The inductive step for our primary induction is now done for finite $\kappa$, proving that (b) holds for all finite index sets $K$.

($\kappa$ infinite).  Identify $K$ with the set of ordinals less than $\lambda$, where $\lambda$ is the first ordinal with cardinality $\kappa$.  Let $\approx_0$ be the equality relation on $N$, and for nonzero $j \in K$, let $\approx_j$ be the smallest congruence on $N$ such that $\sr^+_{N/{\approx_j}}(\pi_{\approx_j}(b_k)) \le m_k$ for all $k \in K$ with $k<j$.  Set $\approx$ equal to the union of the $\approx_j$ for $j \in K$.  Clearly $\approx_i\ \subseteq\ \approx_j\ \subseteq\ \equiv$ for $0 \le i \le j < \lambda$, whence $\approx$ is a congruence on $N$ contained in $\equiv$.  Again, it is straightforward to show that $\sr^+_{N/{\approx}}(\pi_{\approx}(b_k)) \le m_k$ for all $k \in K$.  Thus, $\approx$ equals $\equiv$.

To establish (b) for the current index set $K$, it suffices to show that if $u,v \in N$ and $u \approx_j v$ for some $j \in K$, then there exists $w \in \sum_{k\in K} \Znn\,b_k$ such that $w+u = w+v$.  If $j=0$, then $u=v$ and we are done.  If $j>0$, then the set $K'$ of ordinals less than $j$ has cardinality less than $\kappa$, and our inductive step provides an element $w \in \sum_{k\in K'} \Znn\,b_k$ such that $w+u = w+v$. 

This completes the primary induction and the proof of (b).
\end{proof}

\begin{proposition}  \label{prop3}
Let $\phi : M \rightarrow N$ be a homomorphism of monoids, $(a_k)_{k\in K}$ a nonempty family of elements of $M$, and $(m_k)_{k\in K}$ a corresponding family of positive integers.  Let $\equiv$ be the smallest congruence on $N$ such that $\sr^+_{N/{\equiv}}(\pi_{\equiv}(\phi(a_k))) \le m_k$ for all $k \in K$.

{\rm(a)} If $\phi$ is weakly unitary, then $\pi_{\equiv} \phi$ is weakly unitary.

Now assume that $\phi$ is unitary, and that $\sr^+_M(a_k) \le m_k$ for all $k \in K$.

{\rm(b)} $\pi_{\equiv} \phi$ is unitary.

{\rm(c)}  If $N$ is simple, then $N/{\equiv}$ is simple.

{\rm(d)} If $N$ is conical, then $N/{\equiv}$ is conical.
\end{proposition}

\begin{proof}
Let us abbreviate $N' := N/{\equiv}$ and $\pi := \pi_{\equiv}$, and set $b_k := \phi(a_k)$ for $k \in K$.

(a) Suppose that $u,v \in M$ and $z' \in N'$ with $\pi\phi(u) + z' = \pi\phi(v)$.  Fix some $z \in N$ such that $\pi(z) = z'$; then $\phi(u) + z \equiv \phi(v)$.  By Lemma \ref{lem2}(b), there is some $w \in \sum_{k\in K} \Znn\,b_k$ such that $w+ \phi(u) + z = w+ \phi(v)$.  In particular, $w \in \phi(M)$, and hence the weak unitarity of $\phi$ implies that $z \in \phi(M)$.  Therefore $z' \in \pi\phi(M)$.

For the remainder of the proof, assume that $\phi$ is unitary, and that $\sr^+_M(a_k) \le m_k$ for all $k \in K$.

(b) Since $\phi(M)$ is cofinal in $N$, it follows immediately that $\pi\phi(M)$ is cofinal in $N'$.  Further, $\pi\phi$ is weakly unitary by (a).  Hence, we just need to show that $\pi\phi$ is injective.  We proceed by transfinite induction on $\kappa := \card(K)$.

Suppose first that $\kappa=1$, and take $K = \{1\}$.  Let $u,v \in M$ such that $\pi\phi(u) = \pi\phi(v)$.  By Lemma \ref{lem1}(b), either $u=v$ or there exist a positive integer $k$ and a sequence of elements $u_0= \phi(u) ,u_1,\dots, u_t= \phi(v)$ in $N$ such that for each $j\in [1,t]$, one of $(u_{j- 1},u_j)$ or $(u_j,u_{j-1})$ equals $((m_1-1)b_1 +x_j,y_j)$ for some $x_j,y_j \in N$ with $(k+m_1-1)b_1 + x_j = kb_1 + y_j$.  Assume that $u\ne v$.  After removing any repeated terms from the sequence of $u_j$, we may assume that $t > 0$ and $u_{j-1} \ne u_j$ for $j \in [1,t]$.  In particular, $u_0 \ne u_1$ and $kb_1 + u_0 = kb_1 + u_1$.  In other words, $\phi(ka_1 + u) = \phi(ka_1) + u_1$.  The unitary assumption guarantees that $u_1\in \phi(M)$, say $u_1 = \phi(u'_1)$.

We will only consider the case when $(u_0,u_1) = ((m_1-1)b_1 +x_1,y_1)$, as the other case is very similar.  The unitary hypothesis guarantees that $x_1 = \phi(x'_1)$ for some $x'_1 \in M$.  Consequently, $u = (m_1-1)a_1 + x'_1$.  Since $(k+m_1-1)b_1 + x_1 = kb_1 + y_1$, we also have $(k+m_1-1)a_1 + x'_1 = ka_1 + u'_1$.  Then, the strong $m_1$-stable rank condition on $a_1$ in $M$ tells us that $u =u'_1$.  But this means $u_0 = u_1$, a contradiction.  Therefore $u=v$ in this case, proving that $\pi\phi$ is injective when $\kappa = 1$.

Now suppose that $\kappa \ge 2$ and that the proposition holds for nonempty families of fewer than $\kappa$ elements.  If $\kappa$ is finite, express $\equiv\; =\, \bigcup_{j=0}^\infty \approx_j$ as in the proof of Lemma \ref{lem2}(b).  By induction, injectivity of $\pi_{\equiv_j} \phi$ implies injectivity of $\pi_{\equiv_{j+1}} \phi$ for all $j \ge 0$.  Consequently, $\pi \phi$ is injective.

When $\kappa$ is infinite, identify $K$ with the set of ordinals less than the first ordinal of cardinality $\kappa$, and express $\equiv\; = \bigcup_{j \in K} \approx_j$ as in the proof of Lemma \ref{lem2}(b).  By induction, $\pi_{\equiv_j} \phi$ is injective for all $j \in K$, forcing $\pi \phi$ to be injective in this case.

This concludes the induction, proving that $\pi\phi$ is injective, and thus unitary.

(c) Assuming $N$ is simple, it contains an element $x$ that is not a unit.  Since $\phi(M)$ is cofinal in $N$, there is some $y \in M$ such that $x \le \phi(y)$.  If $y$ is a unit in $M$, then $y \le 0$ in $M$, whence $x \le \phi(y) \le 0$ in $N$, contradicting our assumption on $x$.  Thus, $y$ is not a unit in $M$.  If $\pi\phi(y)$ is a unit in $N'$, say with inverse $z$, then from $\pi\phi(y) + z = \pi\phi(0)$ and unitarity of $\pi\phi$ we obtain $z = \pi\phi(z')$ for some $z' \in M$.  But then $y+z' = 0$, contradicting the fact that $y$ is a non-unit in $M$.  Thus, $\pi\phi(y)$ is not a unit in $N'$, which shows that $N'$ is not a group.

Given an ideal $I$ of $N'$ different from the group of units, choose a 
non-unit $c \in I$, and write $c = \pi(b)$ for some non-unit $b \in N$.  Since $N$ is simple, $b$ is an order-unit in $N$.  Consequently, $c$ is an order-unit in $N'$ and thus $I = N'$.  Therefore $N'$ is simple.

(d) Suppose $x,y \in N'$ with $x+y = 0$.  Write $x = \pi(x')$ and $y = \pi(y')$ for some $x',y' \in N$, so that $x'+y' \equiv 0$.  By Lemma \ref{lem2}(b), there is some $w \in \sum_{k\in K} \Znn\,b_k$ such that $w+x'+y' = w$.  Since $w \in \phi(M)$, unitarity of $\phi$ implies that $x'+y' = \phi(z)$ for some $z \in M$.  Then $\pi\phi(z) = \pi\phi(0)$, whence $z=0$ and consequently $x'+y' = 0$.  Since $N$ is conical, $x' = y' = 0$, and therefore $x = y = 0$.
\end{proof}

\begin{theorem}  \label{simple.crm.keep.sr+.unitary}
Let $M_0$ be a simple conical monoid.  Then there exist a simple conical refinement monoid $N$ and a unitary embedding $\phi : M_0 \rightarrow N$ such that 
$$
\sr^+_N(\phi(a)) = \sr^+_{M_0}(a) \qquad \forall\; a \in M_0 \,.
$$
Consequently,
$$
\sr_{M_0}(a) \le \sr_N(\phi(a)) \le \sr_{M_0}(a) + 1 \qquad \forall\; a \in M_0 \,.
$$
\end{theorem}

\begin{proof}
Once the first statement is proved, the second follows via Lemma \ref{unitary.sr.incr} and \eqref{sr.vs.sr+}.

Set $F := \{ a \in M_0 \mid \sr^+_{M_0}(a) < \infty \}$.  In particular, $0 \in F$.  Choose a set $K$ and a surjective map $k \mapsto a_k$ from $K$ onto $F$, and set $m_k := \sr^+_{M_0}(a_k) \in \Zpos$ for $k \in K$.  In view of Observation \ref{injec.sr+.incr}, it suffices to find a unitary embedding $\phi$ of $M_0$ into a simple conical refinement monoid $N$ such that $\sr^+_N(\phi(a_k)) \le m_k$ for all $k \in K$.

We recursively construct a sequence of monoid homomorphisms\begin{equation}  \label{MNM}
M_0 \overset{\phi_0}{\longrightarrow} N_0 \overset{\pi_0}{\longrightarrow} M_1 \overset{\phi_1}{\longrightarrow} N_1 \overset{\pi_1}{\longrightarrow} M_2 \longrightarrow \cdots 
\end{equation}
such that the following hold for all $i \ge 0$:
\begin{enumerate}
\item $M_{i+1}$ is a simple conical monoid;
\item $N_i$ is a simple conical refinement monoid;
\item $\phi_i$ and $\pi_i \phi_i$ are unitary;
\item $\sr^+_{M_{i+1}}( \pi_i \phi_i \cdots \pi_0 \phi_0 (a_k)) \le m_k$ for all $k \in K$.
\end{enumerate}

To start, choose a simple conical refinement monoid $N_0$ and a unitary embedding $\phi_0 : M_0 \rightarrow N_0$, using \cite[Corollary 2.7]{Wembed}.  Then let $\equiv_0$ be the smallest congruence on $N_0$ such that $\sr^+_{N_0/{\equiv_0}}(\pi_{\equiv_0} \phi_0(a_k)) \le m_k$ for all $k \in K$.  Set $M_1 := N_0/{\equiv_0}$ and let $\pi_0 := \pi_{\equiv_0} : N_0 \rightarrow M_1$.  Proposition \ref{prop3} shows that $M_1$ is simple and conical and $\pi_0 \phi_0$ is unitary.  Moreover, $\sr^+_{M_1}( \pi_0\phi_0 (a_k)) \le m_k$ for all $k \in K$ by the choice of $\equiv_0$.  Thus, (1)--(4) hold for $i=0$.  

Now repeat this process recursively.

Having set up \eqref{MNM}, let $N$ be a direct limit of this sequence, with limit maps $\mu_i : M_i \rightarrow N$ and $\nu_i : N_i \rightarrow N$.  Moreover, set 
$$
\psi_i := \pi_{i-1} \phi_{i-1} \cdots \pi_0 \phi_0 : M_0 \rightarrow M_i \qquad \forall\; i\ge0.
$$
In view of (3), all of the maps $\psi_i$ are unitary.  It follows that $\mu_0$ is unitary, which implies in particular that $N$ is nonzero.
Since all of the $N_i$ are simple and conical with refinement, $N$ is a simple conical refinement monoid.  Thus, taking $\phi := \mu_0$, we are almost done.

It remains to show that $\sr^+_N(\mu_0(a_k)) \le m_k$ for all $k \in K$.  Suppose $k \in K$ and $m_k \mu_0(a_k) + x = \mu_0(a_k) + y$ for some $x,y \in N$.  There is an index $i$ such that $x = \mu_i(x')$ and $y = \mu_i(y')$ for some $x',y' \in M_i$, and after increasing $i$ suitably we may assume also that $m_k \psi_i(a_k) + x' = \psi_i(a_k) + y'$.  Since $\sr^+_{M_i}(\psi_i(a_k)) \le m_k$ by (4), it follows that $(m_k-1) \psi_i(a_k) + x' = y'$.  Applying $\mu_i$, we conclude that $(m_k-1) \mu_0(a_k) + x = y$, proving that $\sr^+_N(\mu_0(a_k)) \le m_k$.
\end{proof}

The second statement of Theorem \ref{simple.crm.keep.sr+.unitary} cannot be reduced to a general equality.  On one hand, $\sr^+ = 1$ is the same as cancellativity, so if $M_0$ is cancellative, the theorem yields $\sr_{M_0}(a) = 1 = \sr_N(\phi(a))$ for all $a \in M_0$.  On the other hand, if we apply the theorem with $M_0$ equal to the monoid $M$ of Example \ref{arch.sr2&4}, we cannot have both $\sr_N(\phi(a)) = \sr_{M_0}(a) = 4$ and $\sr_N(\phi(2a)) = \sr_{M_0}(2a) = 2$ because of Theorem \ref{sr.ka.refine}.  In fact, $\sr^+_N(\phi(a)) = \sr^+_{M_0}(a) = 5$, whence $\sr_N(\phi(a))$ is either $4$ or $5$, and so $\sr_N(\phi(2a)) = 3$ by Theorem \ref{sr.ka.refine}.

We can now show that the third case of the trichotomy established in Theorem \ref{trichot.simple.crm} occurs.

\begin{theorem}  \label{simple.crm.sr[2,infty)}
There exist simple conical refinement monoids $N$ such that
\begin{equation}  \label{N.sr2infty}
\sr_N(N \setminus \{0\}) = \ZZ_{\ge2} \,.
\end{equation}
\end{theorem}

\begin{proof}
Let $M_0$ be one of the monoids of Example \ref{exist.sra=?}(2), for some integer $n \ge 2$, with the generator $a$ having stable rank $n$.  It is clear from the construction that $M_0$ is simple and conical.  Applying Theorem \ref{simple.crm.keep.sr+.unitary}, we obtain a simple conical refinement monoid $N$ and a unitary embedding $\phi : M_0 \rightarrow  N$ such that $n \le \sr_N(\phi(a)) \le n+1$.  Since $n \ge 2$, \eqref{N.sr2infty} follows from Theorem \ref{trichot.simple.crm}.
\end{proof}

Once a monoid with the properties of Theorem \ref{simple.crm.sr[2,infty)} is in hand, it can be cut down to a countable monoid with the same properties by standard procedures, as follows.

\begin{corollary}  \label{ctbl.sr[2,infty)}
There exist countable simple conical refinement monoids $C$ such that
$$
\sr_C(C \setminus \{0\})  = \ZZ_{\ge2} \,.
$$
\end{corollary}

\begin{proof}
Choose $N$ as in Theorem \ref{simple.crm.sr[2,infty)}.  For each nonzero $a \in N$,  set $n_a := \sr_N(a)$, and choose $x_a,y_a \in N$ such that
\begin{itemize}
\item $(n_a-1)a+x_a = a+y_a$ but $\nexists$ $e \in N$ with $(n_a-1)a = a+e$ and $e+x_a = y_a\,$.
\end{itemize}
Also, set $X_a := \{ (x,y) \in N^2 \mid n_a a+x = a+y \}$, and for each $(x,y) \in X_a$, choose $e(x,y) \in N$ such that
\begin{itemize}
\item $n_a a = a + e(x,y)$ and $e(x,y)+x = y$.
\end{itemize}

We will build $C$ by repeating several basic steps, the first of which is
\begin{enumerate}
\item For any countable submonoid $K \subseteq N$, there is a countable submonoid $L \subseteq N$ such that $K \subseteq L$ and $\sr_L(a) = \sr_N(a)$ for all $a \in L$.
\end{enumerate}
To prove (1), construct countable submonoids $K_0 := K \subseteq K_1 \subseteq \cdots$ of $N$ where
\begin{itemize}
\item $K_{i+1}$ is the submonoid of $N$ generated by $K_i$ together with $\{ x_a,y_a \mid a \in K_i \}$ and $\{ e(x,y) \mid (x,y) \in X_a \cap K_i^2 \}$.
\end{itemize}
Then $L := \bigcup_{i=0}^\infty K_i$ is a countable submonoid of $N$ satisfying (1).

By similar means, we see that
\begin{enumerate}
\item[(2)] For any countable submonoid $K \subseteq N$, there is a countable submonoid $L \subseteq N$ such that $K \subseteq L$ and $L$ has refinement.
\item[(3)] For any countable submonoid $K \subseteq N$, there is a countable submonoid $L \subseteq N$ such that $K \subseteq L$ and $L$ is simple.
\end{enumerate}

Our final construction consists of cycling through (1), (2), and (3) countably many times.  To start, choose $a_2,a_3, \ldots \in N$ such that $\sr(a_k) = k$ for each $k \in \ZZ_{\ge2}$, and let $C_0$ be the submonoid of $N$ generated by $\{a_2,a_3,\dots \}$.  Then construct countable submonoids $C_0 \subseteq C_1 \subseteq \cdots$ of $N$ such that for all $i \in \Znn$,
\begin{itemize}
\item $\sr_{C_{3i+1}}(a) = \sr_N(a)$ for all $a \in C_{3i+1}$.
\item $C_{3i+2}$ has refinement.
\item $C_{3i+3}$ is simple.
\end{itemize}
Then $C := \bigcup_{j=0}^\infty C_j$ is a countable submonoid of $N$ with the desired properties.
\end{proof}

\sectionnew{Monoids of projective modules}  \label{V(R)s}

We discuss monoids built from isomorphism classes of projective modules, with addition induced from direct sums, and consider stable ranks within these monoids.  Monoids built from more general classes of modules will be discussed in the following section.

All rings mentioned are assumed to be associative, and unital unless otherwise indicated.  

\begin{definition}  \label{def.vr}
Let $R$ be a ring and $\FP(R)$ the class of finitely generated projective right $R$-modules.  For each $A \in \FP(R)$, let $[A]$ be a label for the \emph{isomorphism class of $A$}.  (These isomorphism classes are proper classes, so they cannot be members of a set, but we can choose a set of labels for them, since there are sub\emph{sets} $\FP_0(R) \subset \FP(R)$ such that each module in $\FP(R)$ is isomorphic to exactly one module in $\FP_0(R)$.)  Then define
$$
V(R) := \{ [A] \mid A \in \FP(R) \}.
$$
There is a well-defined addition operation on $V(R)$ induced from the direct sum operation:
$$
[A] + [B] := [A \oplus B] \qquad \forall\; A,B \in \FP(R).
$$
With this operation, $V(R)$ becomes a conical commutative monoid, and $[R]$ is an order-unit in $V(R)$.  The algebraic order in $V(R)$ is given by the following rule:
$$
[A] \le [B] \ \iff \ A\; \text{is isomorphic to a direct summand of}\; B.
$$

One may equally well build $V(R)$ from the class $\FP_\ell(R)$ of finitely generated projective left $R$-modules, since the functors $\Hom_R(-,R)$ restrict to equivalences between the full subcategories of $\Modr$ and $\rMod$ generated by $\FP(R)$ and $\FP_\ell(R)$ that preserve and reflect isomorphisms and direct sums.
\end{definition}

Recall that a ring $R$ is (\emph{von Neumann}) \emph{regular} if, for each $a \in R$, there is some $x \in R$ satisfying $axa = a$.  It is an \emph{exchange ring} if the regular left module $_RR$ satisfies the finite exchange property in direct sums of modules.  (By \cite[Corollary 2]{Wexch}, this condition is left-right symmetric.)  Every regular ring is an exchange ring \cite[Theorem 3]{Wexch}.  

\begin{definition}  \label{sr.ring}
For any ring $S$, we let $\sr(S)$ denote the K-theoretic \emph{stable rank} of $S$, which is the least positive integer in the stable range of $S$ (if such integers exist) or $\infty$ (otherwise).  A positive integer $n$ lies \emph{in the stable range of $S$} provided that for any left unimodular row $(s_1,\dots,s_{n+1}) \in S^{n+1}$ (meaning that $\sum_{i=1}^{n+1} Ss_i = S$), there are elements $a_i \in S$ such that the row $(s_1+a_1s_{n+1}, \dots, s_n+a_ns_{n+1}) \in S^n$ is left unimodular.  This condition is left-right symmetric by \cite[Theorem 2]{Vas}.
\end{definition}

\begin{theorem}  \label{exch->ref}
Let $R$ be an exchange ring.

{\rm(a)} {\rm\cite[Corollary 1.3]{AGOP}} $V(R)$ is a refinement monoid.

{\rm(b)} {\rm\cite[Theorem 3.2]{AGOP}} $\sr_{V(R)}([A]) = \sr(\End_R(A))$ for each $A \in \FP(R)$.
\end{theorem} 

When $R$ is an exchange ring, Theorem \ref{exch->ref}(b) combined with Vaserstein's theorem \cite[Theorem 3]{Vas} implies that if $a \in V(R)$ with $\sr_{V(R)}(a) = n < \infty$, then
$$
\sr_{V(R)}(ka) = 1 + \left\lceil \frac{n-1}{k} \right\rceil \qquad \forall\; k \in \Zpos\,,
$$
matching Theorem \ref{sr.ka.refine}.
Consequently, a monoid with the properties of Example \ref{arch.sr2&4} cannot be isomorphic to $V(R)$ for any exchange ring $R$.

For any ring $R$, the separativity condition in $V(R)$ translates to modules in the form:
\begin{equation}  \label{sep.rings}
A \oplus A \cong A \oplus B \cong B \oplus B \ \implies \ A \cong B, \qquad \forall\; A,B \in \FP(R).
\end{equation}
We say that $R$ is \emph{separative} provided \eqref{sep.rings} holds.
\begin{problems}  \label{sp&rp}
Separativity is a longstanding open question for regular rings and exchange rings:
\begin{itemize}
\item[] \emph{The Separativity Problem}: Is every regular (resp., exchange) ring separative?
\end{itemize}

To test this and many other questions, one would like to know which monoids can be realized as $V(R)$s for regular or exchange rings $R$.  Any such monoid must be conical and have refinement and an order-unit, so the question was first raised in terms of those properties alone.  However, Wehrung then constructed examples of conical refinement monoids (with order-units) that cannot be realized as $V(R)$ for any regular ring $R$ \cite[Corollary 2.12 and remark following]{Wnonmeas}.  (It is unknown whether these examples can be realized as $V(R)$ for some exchange rings $R$.)  On the other hand, separativity and many other questions can be reduced to countable monoids and countable regular or exchange rings, and the question remains open in countable cases:
\begin{itemize}
\item[] \emph{The Realization Problem}: Is every countable conical refinement monoid with an order-unit isomorphic to $V(R)$ for some exchange (or regular) ring $R$?
\end{itemize}

These two problems are inextricably linked:
\begin{itemize}
\item[] \emph{The Separativity Problem and the Realization Problem cannot both have positive answers.}
\end{itemize}
This is due to the existence of monoids that are countable, conical, have refinement and order-units, but are not separative, as discussed in \cite[Section 1]{Ara.realization.survey}.  If such a monoid is isomorphic to $V(R)$ for an exchange ring $R$, then $R$ is not separative.  On the other hand, if all regular (resp., exchange) rings are separative, then a monoid with the above properties cannot be realized as $V(R)$ for a regular (resp., an exchange) ring $R$.

For instance, let $M$ be one of the monoids in Example \ref{exist.sra=?}(2), and observe that $M$ is simple and conical.  It is not separative because $2(n-1)a = 2b = (n-1)a+b$ while $(n-1)a \ne b$.  By \cite[Corollary 2.7]{Wembed}, $M$ embeds into a simple conical refinement monoid $N$.  Note that simplicity implies that $a$ is an order-unit in $N$.  Then there is a countable submonoid $M^+$ of $N$ that contains (the image of) $M$, has refinement, and in which $a$ is an order-unit.  Since $N$ is conical, so is $M^+$.  Since $M$ is not separative, neither is $M^+$.

Another example is given by Corollary \ref{ctbl.sr[2,infty)}: a countable simple conical refinement monoid $C$ such that $\sr_C(C\setminus\{0\}) = \ZZ_{\ge2}$.  By Theorem \ref{sep.trichot.arch}, $C$ cannot be separative.

An additional advantage of this second example is that if $C \cong V(R)$ for some regular ring $R$, then elements $a_k \in C$ with $\sr_C(a_k) = k$ correspond to finitely generated projective $R$-modules $A_k$ such that $\End_R(A_k)$ is a simple regular ring with stable rank $k$, for each integer $k \ge 2$.
\end{problems}

\begin{remark}  \label{extns.sep.via.monoids}
An obstacle to possible constructions of non-separative exchange rings is the following Extension Theorem \cite[Theorem 4.2]{AGOP}:  If $R$ is an exchange ring, $I$ an ideal of $R$ that is separative in a suitable non-unital sense, and $R/I$ is separative, then $R$ is separative.  No module-theoretic proof of this theorem has been found; it is derived from a corresponding extension theorem for refinement monoids \cite[Theorem 4.5]{AGOP}.
\end{remark}

An application of our current monoid results yields the following:

\begin{theorem}  \label{match.sr.corners}
Let $R$ be a separative exchange ring.  If $e$ and $f$ are any idempotents in $R$ such that $ReR = RfR$, then $\sr(eRe) = \sr(fRf)$.  In particular, if $ReR = R$, then $\sr(eRe) = \sr(R)$.
\end{theorem}

\begin{proof}
By Theorem \ref{exch->ref}(b), $\sr(eRe) = \sr_{V(R)}([eR])$ and $\sr(fRf) = \sr_{V(R)}([fR])$.

Since $e \in RfR$, we must have $e = x_1fy_1 + \cdots + x_nfy_n$ for some elements $x_i \in  eRf$ and $y_i \in fRe$.  Consequently, there is a surjective $R$-module homomorphism $\phi : (fR)^n \rightarrow eR$ given by the rule $\phi(r_1,\dots,r_n) = x_1r_1 + \cdots + x_n r_n$.  Since $\phi$ splits by projectivity of $fR$, we find that $eR$ is isomorphic to a direct summand of $(fR)^n$, and so $[eR] \le n[fR]$ in $V(R)$.  By symmetry, $[fR] \le m[eR]$ for some $m \in \Zpos$, and thus $[eR]$ and $[fR]$ lie in the same archimedean component of $V(R)$.  

Since $V(R)$ is separative, Theorem \ref{sep.trichot.arch} implies that $\sr_{V(R)}([eR]) = \sr_{V(R)}([fR])$.
\end{proof}

Concerning general rings, we cite a theorem of Bergman \cite[Theorem 6.4]{Ber}, as corrected and extended by Bergman and Dicks \cite[Remarks following Theorem 3.4]{BD}, which states that any conical commutative monoid with an order-unit is isomorphic to $V(R)$ for some hereditary algebra $R$ over a pre-chosen field $K$.  (The order-unit condition can be dropped if $R$ is allowed to be non-unital \cite[Corollary 4.5]{AGsepgraph}.)  Consequently, the monoids $V(R)$ for arbitrary rings $R$ cannot satisfy any less-than-general conical monoid properties.

\begin{theorem}  \label{srVR.le.srEnd}
If $R$ is a ring, then
$$
\sr_{V(R)}([A]) \le \sr(\End_R(A)) \qquad \forall\; A \in \FP(R).
$$
\end{theorem}

\begin{proof}
Let $A \in \FP(R)$, and assume that $\sr(\End_R(A)) = n < \infty$. Suppose we have some $X,Y \in \FP(R)$ with $n[A]+[X] = [A]+[Y]$, so that $A^n \oplus X \cong A \oplus Y$.  By \cite[Theorem 1.6]{War}, $A$ satisfies the \emph{$n$-substitution property} of \cite[Definition 1.1]{War}, and so \cite[Theorem 1.3]{War} implies that there is a direct summand $L$ of $A^n$ such that $Y \cong X \oplus L$ and $L \oplus A \cong A^n$.  Hence, $n[A] = [A]+[L]$ and $[L]+[X] = [Y]$.  This proves that $\sr_{V(R)}([A]) \le n$.
\end{proof}

The inequality in Theorem \ref{srVR.le.srEnd} is usually strict.  For instance, let $R = k[x_1,\dots,x_n]$ be a polynomial ring over a field $k$.  By the Quillen-Suslin Theorem, all finitely generated projective modules over $R$ are free, and so $V(R) \cong \Znn$.  As a result, $\sr_{V(R)}([A]) = 1$ for all $A \in \FP(R)$, and, in particular, $\sr_{V(R)}([R]) = 1$

On the other hand, the stable rank of $\End_R(R) \cong R$ can be arbitrarily large (but finite).  For instance, if $k$ is a subfield of $\RR$ then $\sr(R) = n+1$ as shown by Vaserstein \cite[Theorem 8]{Vas}.

\begin{example}  \label{example.swan}
Let $R := \RR[x_0,\dots,x_n]/ \langle x_0^2+ \cdots + x_n^2 -1 \rangle$, where $n$ is a positive integer different from $1$, $3$, $7$.  Then $\sr_{V(R)}([R]) = n+1$, as follows.

On one hand, since $R$ has Krull dimension $n$, a theorem of Bass \cite[Theorem 11.1]{Bas} says that $\sr(R) \le n+1$.  On the other hand, by \cite[Theorem 3]{Swa} there is a projective $R$-module $P$ such that $P \oplus R \cong R^{n+1}$ but $P$ is not free.  Then $[R]+n[R] = [R]+[P]$ in $V(R)$, but $n[R] \ne [P]$, so Lemma \ref{W1.2} implies that $\sr_{V(R)}([R]) \nleq n$.  Therefore $\sr_{V(R)}([R]) = n+1$. 

In particular, since $n\ge2$, the trichotomy of Corollary \ref{trichot} now shows that $V(R)$ is not separative.  This also follows from the facts that $2n[R] = n[R]+[P] = 2[P]$, where the first equality is immediate from $(n+1)[R] = [R]+[P]$ and the second holds because $\CC\otimes_\RR P \cong (\CC\otimes_\RR R)^n$ as modules over $\CC\otimes_\RR R$ \cite[Remark, p.270]{Swa}.

In case $n$ is even, the module $P$ is also indecomposable \cite[Theorem 3]{Swa}.  Consequently, the equality $[R]+[P] = [R]+n[R]$ cannot be refined in $V(R)$, and therefore $V(R)$ is not a refinement monoid.
\end{example}

\sectionnew{Monoids of general modules}  \label{V(C)s}

The construction of $V(R)$ can obviously be applied to classes of modules other than $\FP(R)$, and equivalence relations other than isomorphism can be used.  We discuss some of these monoids in the present section. 

\begin{definition}  \label{def.V(C)}
Suppose $\calC$ is a class of modules (say right modules) over a ring $R$, closed under isomorphisms and finite direct sums, and containing the zero module (or, a corresponding class of objects in a category with finite coproducts and a zero object).  Assume also that $\calC$ is \emph{essentially small} (or \emph{skeletally small}), meaning that there is a sub\emph{set} $\calC_0$ of $\calC$ such that each module in $\calC$ is isomorphic to exactly one module in $\calC_0$.  For $A \in \calC$, let $[A] := [A]_{\calC}$ be a label for the isomorphism class of $A$.  Exactly as in Definition \ref{def.vr}, we set
$$
V(\calC) := \{ [A] \mid A \in \calC \},
$$
and we define
$$
[A] + [B] := [A \oplus B] \qquad \forall\; A,B \in \calC.
$$
Then $V(\calC)$ becomes a conical commutative monoid, with zero element $[0]$, but it may or may not have an order-unit.  (For instance, if $\calC$ is the class of all finite abelian groups, there is no order-unit in $V(\calC)$.)  Note that for $[A],[B] \in V(\calC)$, we have
$$
[A] \le [B] \ \text{in} \ V(\calC) \ \iff \ \exists\; X \in \calC \ \text{such that} \ A \oplus X \cong B.
$$
If $\calC$ is closed under direct summands (within $\Modr$), we have $[A] \le [B]$ if and only if $A$ is isomorphic to a direct summand of $B$.
\end{definition}

A natural choice of a class $\calC$ as above is the class $\FG(R)$ of all finitely generated right $R$-modules.  There is no standard notation for the monoid $V(\FG(R))$.

The same argument used to prove Theorem \ref{srVR.le.srEnd} also shows that:

\begin{theorem}  \label{srVC.le.srEnd}
Let $\calC$ be an essentially small class of modules over some ring $R$, closed under isomorphisms, finite direct sums and direct summands, and containing the zero module.  Then
$$
\sr_{V(\calC)}([A]) \le \sr(\End_R(A)) \qquad \forall\; A \in \calC.
$$
\end{theorem} 

For certain classes $\calC$ of modules, there are known finite upper bounds for the sets $\sr(V(\calC))$.  We recall that a commutative ring $S$ is called \emph{J-noetherian} if it satisfies the ascending chain condition for semiprimitive ideals (i.e., ideals $I$ for which $J(S/I) = 0$), and that the \emph{J-dimension} of $S$ is the supremum of the lengths of chains of semiprimitive prime ideals of $S$.  Bass proved in \cite[Theorem 11.1]{Bas} that if $S$ is a commutative J-noetherian ring of J-dimension $d < \infty$ and $R$ is a module-finite $S$-algebra (meaning that $R$ is finitely generated as an $S$-module), then $\sr(R) \le d+1$.  Warfield extended this to endomorphism rings of finitely presented $R$-modules, from which we obtain

\begin{theorem}  \label{bound.srC.from.Warfield}
Let  $S$ be a commutative J-noetherian ring of J-dimension $d < \infty$, and $R$ an $S$-algebra such that for each semiprimitive prime ideal $P$ of $S$, the localization $R_P$ is a module finite $S_P$-algebra.  Let $\calC$ be the class of finitely presented right $R$-modules (or any subclass closed under isomorphisms, finite direct sums and direct summands, and containing the zero module).  Then
$$
\sr_{V(\calC)}([A]) \le d+1 \qquad \forall\; A \in \calC.
$$
\end{theorem}

\begin{proof}
By \cite[Theorem 3.4]{War}, $\sr(\End_R(A)) \le d+1$ for all $A \in \calC$.  The result thus follows from Theorem \ref{srVC.le.srEnd}.
\end{proof}

Example \ref{example.swan} provides instances in which the upper bound $d+1$ of Theorem \ref{bound.srC.from.Warfield} is attained, for any positive integer $d \ne 1,3,7$.

\begin{example}  \label{V(T)}
Let $\calT$ denote the class of torsion-free abelian groups of finite rank.  (It is essentially small because every group in $\calT$ is isomorphic to a subgroup of one of the vector spaces $\QQ^n$, $n \in \Zpos$.)  It is well known that many types of cancellation fail to hold in $\calT$, whence the analogs also fail in $V(\calT)$.  For instance:
\begin{itemize}
\item There exist $A,X,Y \in \calT$ such that $A \oplus X \cong A \oplus Y$ but $X \ncong Y$ \cite[Section 2]{Jon} (also \cite[Example 2.10]{Arn}).
\item There exist $A,Y \in \calT$ such that $A \oplus A \cong A \oplus Y$ but $A \ncong Y$ \cite[Example 8.20]{Arn}.
\item There exist $A,B \in \calT$ such that $A\oplus A \cong A\oplus B \cong B \oplus B$ but $A \ncong B$ \cite[Theorem 12]{OV}.
\item There exist $A,B \in \calT$ such that $A^n \cong B^n$ for all integers $n\ge2$ but $A \ncong B$ \cite[Theorem 12]{OV}.
\end{itemize}

In particular, $V(\calT)$ is not separative.
It is also known that $V(\calT)$ does not have refinement.  This follows from an example of J\'onsson \cite{Jon1}, which provides pairwise non-isomorphic indecomposable groups $A,B,C,D \in \calT$ such that $A\oplus B \cong C \oplus D$.  Since $A$, $C$, $D$ are indecomposable and $A \ncong C,D$, there is no decomposition $A \cong A_1 \oplus A_2$ such that $A_1$ is isomorphic to a direct summand of $C$ and $A_2$ is isomorphic to a direct summand of $D$.  Consequently, the equation $[A]+[B] = [C]+[D]$ in $V(\calT)$ cannot be refined.

Warfield proved that $\sr(\End(A)) \le 2$ for all $A \in \calT$ \cite[Theorem 5.6]{War}, which by Theorem \ref{srVC.le.srEnd} implies that $\sr_{V(\calT)}([A]) \le 2$ for all such $A$.  (This conclusion, in group-theoretic form, was also proved in \cite[Theorem 5.6]{War}.)  Since ``most" $A \in \calT$ do not cancel from direct sums, $\sr_{V(\calT)}([A]) = 2$ for ``most" $A \in \calT$.  However, there do exist torsion-free abelian groups $A$ of finite rank which are cancellative, starting with $A = \ZZ$ (e.g., \cite[Corollary 8.8(b)]{Arn}).

Warfield also noted, in \cite[Remark, p.479]{War}, that there exist groups $A \in \calT$ such that (in our terminology) $[A]$ is not Hermite in $V(\calT)$.  
\end{example}

Monoids analogous to $V(\calC)$ may also be constructed using relations coarser than isomorphism, as follows.

\begin{definition}  \label{def.V(C)/sim}
Let $\calC$ be an essentially small class of modules over some ring $R$, closed under isomorphisms, finite direct sums and direct summands, and containing the zero module.   Suppose we have an equivalence relation $\sim$ on $\calC$ that is \emph{stable under isomorphism} and \emph{stable under additional summands}, that is, ($A \cong B \implies A \sim B$) and ($A \sim B \implies A \oplus C \sim B \oplus C$) for any $A,B,C \in \calC$.

For each $A \in \calC$, let $[A]_\sim$ be a label for the $\sim$-equivalence class of $A$.  Following the pattern of Definition \ref{def.V(C)}, we set
$$
V(\calC/{\sim}) := \{ [A]_\sim \mid A \in \calC \},
$$
and we define
$$
[A]_\sim + [B]_\sim := [A \oplus B]_\sim \qquad \forall\; A,B \in \calC.
$$
Then $V(\calC/{\sim})$ becomes a commutative monoid.  Unlike $V(\calC)$, however, $V(\calC/{\sim})$ is not necessarily conical.  For instance, let $\calC$ be the class of finite abelian groups, fix an integer $n \ge 3$, and define $\sim$ on $\calC$ by the rule $A \sim B$ if and only if $\card(A) \equiv \card(B) \pmod{n}$.  Then $(\ZZ/(n-1)\ZZ) \oplus (\ZZ/(n-1)\ZZ) \sim \{0\}$ but $\ZZ/(n-1)\ZZ \not\sim \{0\}$.
\end{definition}

\begin{examples}  \label{stC/sim.expls}
The following relations on the class $\calT$ have been intensively studied.  Groups $A$ and $B$ in $\calT$ are
\begin{align*}
&\text{\emph{multi-isomorphic}} &&\text{in case} &&\quad A^n \cong B^n &&\forall\; n \ge 2;  \\
&\text{\emph{stably isomorphic}} &&\text{in case} &&\quad A \oplus C \cong B \oplus C &&\text{for some}\; C \in \calT;  \\
&\text{\emph{near-isomorphic}} &&\text{in case} &&\quad A^n \cong B^n &&\text{for some}\; n > 0;  \\
&\text{\emph{quasi-isomorphic}} &&\text{in case} &&\quad A \cong A' \le B \cong B' \le A.
\end{align*}
(The original definition of near-isomorphism required the existence of a homomorphism $f : A \rightarrow B$ such that the localization $f_p : A_p \rightarrow B_p$ is an isomorphism for all primes $p$. Warfield proved that near-isomorphism is equivalent to the condition given above \cite[Theorem 5.9]{War}.  The original definition of quasi-isomorphism required the existence of mutual embeddings whose cokernels are bounded.  By, e.g., \cite[Corollary 6.2]{Arn}, the boundedness condition is redundant for groups in $\calT$.)

As discussed in \cite[pp.539,540]{OV},
\begin{multline*}
\text{isomorphism} \implies \text{multi-isomorphism} \implies \text{stable isomorphism} \\
\implies \text{near-isomorphism} \implies \text{quasi-isomorphism},
\end{multline*}
and none of these implications is reversible.
It is clear that these relations are equivalence relations, and that they are stable under isomorphisms and additional summands.

(1)  If $si$ is the relation of stable isomorphism on $\calT$, then $V(\calT/si)$ is clearly cancellative, and so we have $\sr_{V(\calT/si)}([A]_{si}) = 1$ for all $A \in \calT$.

(2)  It follows from \cite[Corollary 5.10]{War} that the relation $ni$ of near-isomorphism on $\calT$ cancels from direct sums.  Hence, $V(\calT/ni)$ is cancellative, and so $\sr_{V(\calT/ni)}([A]_{ni}) = 1$ for all $A \in \calT$.

(3)  Proposition \ref{M.mod.multi-isom}, together with the fact that $\sr_{V(\calT)}([A]) \le 2$ for all $A \in \calT$ (recall Example \ref{V(T)}), implies that if $mi$ is the relation of multi-isomorphism on $\calT$, then all elements of $V(\calT/mi)$ are Hermite, and so $\sr_{V(\calT/mi)}([A]_{mi}) \le 2$ for all $A \in \calT$.

As noted in \cite[p.540]{OV}, J\'onsson's example \cite[Section 2]{Jon} (cf.~\cite[Example 2.10]{Arn}) provides groups $A$, $B$, $C$, $D$ in $\calT$ such that $A \cong C$ and $A \oplus B \cong C \oplus D$ but $B$ and $D$ are not multi-isomorphic.  Thus $[A]_{mi}$ is not cancellative in $V(\calT/mi)$, yielding $\sr_{V(\calT/mi)}([A]_{mi}) = 2$ (because $V(\calT/mi)$ is conical).

(4)  The relation $qi$ of quasi-isomorphism is cancellative with respect to direct sums, as follows from the uniqueness of quasi-decompositions into strongly indecomposable groups in $\calT$ up to quasi-isomorphism (e.g., \cite[Corollary 7.9]{Arn}).  In particular, $V(\calT/qi)$ is cancellative.  Thus $\sr_{V(\calT/qi)}([A]_{qi}) = 1$ for all $A \in \calT$.  

More strongly, the uniqueness theorem implies that $V(\calT/qi)$ is a direct sum of copies of $\Znn$, one copy for each quasi-isomorphism class of strongly indecomposable groups in $\calT$.  

The nonzero subgroups of $\QQ$ are certainly strongly indecomposable, and two of them are quasi-isomorphic if and only if isomorphic (e.g., \cite[Corollary 1.3]{Arn}).  Therefore $V(\calT/qi)$ is an infinite direct sum of copies of $\Znn$.  It follows that $V(\calT/qi)$ does not have an order-unit.  Since there exist surjective monoid homomorphisms
$$
V(\calT) \rightarrow V(\calT/mi) \rightarrow V(\calT/si) \rightarrow V(\calT/ni) \rightarrow V(\calT/qi),
$$
none of the monoids $V(\calT)$, $V(\calT/mi)$, $V(\calT/si)$, $V(\calT/ni)$ has an order-unit.
\end{examples}

\begin{example}  \label{Bro.cancel}
Let $\calN$ be the class of all noetherian modules (right modules, say)  over a ring $R$.  Define a relation $\sim$ on $\calN$ by
\begin{itemize}
\item[] $A \sim B$ if and only if $A$ and $B$ have isomorphic submodule series, meaning that there exist chains of submodules $A_0=0 \le A_1 \le \cdots \le A_n = A$ and $B_0=0 \le B_1 \le \cdots \le B_n = B$ together with a permutation $\sigma \in S_n$ such that $A_i/A_{i-1} \cong B_{\sigma(i)}/B_{\sigma(i)-1}$ for all $i=1,\dots,n$.  
\end{itemize}
This is an equivalence relation on $\calN$ \cite[Proposition 3.3]{Bro}, which is clearly stable under isomorphisms and additional summands.  We can thus construct the monoid $V(\calN/{\sim})$.  It is conical, and $[R]_{\sim}$ is an order-unit in $V(\calN/{\sim})$.  By \cite[Proposition 3.8 and Theorem 5.1]{Bro}, $V(\calN/{\sim})$ has refinement and is strongly separative.  Thus, all elements of $V(\calN/{\sim})$ are Hermite (Lemma \ref{agop.p126}) and consequently have stable rank at most $2$.

On the other hand, stable rank $1$ (equivalently, cancellation) can fail in $V(\calN/{\sim})$.  Take $R = \ZZ$ for instance.  As noted at the end of \cite[p.223]{Bro}, $[\ZZ]_{\sim} = [\ZZ/2\ZZ]_{\sim} + [\ZZ]_{\sim}$ but $[0]_{\sim} \ne [\ZZ/2\ZZ]_{\sim}$.  Therefore $\sr_{V(\calN/{\sim})}([\ZZ]_{\sim}) = 2$.

A mixed version of cancellation does hold relative to $\calN$:  if $A \in \calN$ and $X$, $Y$ are arbitrary $R$-modules such that $A \oplus X \cong A \oplus Y$, then $X \sim Y$ \cite[Theorem 5.5]{Bro}.
\end{example}

\section{Acknowledgements}

The first and sixth authors were partially supported by the Spanish State Research Agency (grant No.~PID2020-113047GB-I00/AEI/10.13039/501100011033), by the Comissionat per Universitats i Recerca de la Generalitat de Catalunya (grant No.~2017-SGR-1725) and by the Spanish State Research Agency through the Severo Ochoa and Mar\'ia de Maeztu Program for Centers and Units of Excellence in R\&D (CEX2020-001084-M).  The third author was partially supported by the Simons Foundation (grant \#963435).
The third and fourth authors thank the Mathematics Department of the Universitat Aut\`onoma de Barcelona for hospitality during their visits, and the fourth author also thanks the Mathematics Departments of the University of California, Santa Barbara and the Universidad de C\'adiz for their hospitality during the early part of this project.  The fifth author was partially supported by PAI III grant FQM-298 of the Junta de Andaluc\'ia, by the DGI-MINECO and European Regional Development Fund, jointly, through grant PID2020-113047GB-I00, and by the grant ``Operator Theory: an interdisciplinary approach'', reference ProyExcel 00780, a project financed in the 2021 call for Grants for Excellence Projects, under a competitive bidding regime, aimed at entities qualified as Agents of the Andalusian Knowledge System, in the scope of the Plan Andaluz de Investigaci\'on, Desarrollo e Innovaci\'on (PAIDI 2020), Consejer\'ia de Universidad, Investigaci\'on e Innovaci\'on of the Junta de Andaluc\'ia.


\end{document}